\documentclass{amsart}
\usepackage[utf8]{inputenc}
\usepackage{amsmath}
\usepackage{amsfonts}
\usepackage{amssymb}
\usepackage{float}
\usepackage{tikz}
\usepackage[colorlinks=true]{hyperref}
\hypersetup{urlcolor=blue, citecolor=red}
\usepackage{hyperref}

\numberwithin{equation}{section}

\newtheorem{theorem}{Theorem}[section]
\newtheorem{corollary}[theorem]{Corollary}

\newtheorem{lemma}[theorem]{Lemma}
\newtheorem{proposition}[theorem]{Proposition}

\newtheorem{assumption}[theorem]{Assumption}

\theoremstyle{definition}

\newcommand{\R}{\mathbb{R}}

\newcommand{\N}{\mathbb{N}}
\newcommand{\C}{\mathbb{C}}
\newcommand{\T}{\mathbb{T}}
\newcommand{\Z}{\mathbb{Z}}

\newcommand {\vanish}[1]{\relax}

\DeclareFontFamily{U}{mathx}{}
\DeclareFontShape{U}{mathx}{m}{n}{<-> mathx10}{}
\DeclareSymbolFont{mathx}{U}{mathx}{m}{n}
\DeclareMathAccent{\widehat}{0}{mathx}{"70}
\DeclareMathAccent{\widecheck}{0}{mathx}{"71}

\def\Xint#1{\mathchoice
{\XXint\displaystyle\textstyle{#1}}%
{\XXint\textstyle\scriptstyle{#1}}%
{\XXint\scriptstyle\scriptscriptstyle{#1}}%
{\XXint\scriptscriptstyle\scriptscriptstyle{#1}}%
\!\int}
\def\XXint#1#2#3{{\setbox0=\hbox{$#1{#2#3}{\int}$ }
\vcenter{\hbox{$#2#3$ }}\kern-.6\wd0}}

\def\dashint{\Xint-}
\begin{document}

\title[Weighted~decoupling~inequalities]{Weighted decoupling inequalities via fractal bilinear Kakeya estimates}

\author{Robert Schippa}
\address{Mathematical Institute of the University of Bonn, Endenicher Allee 60, D-53115 Bonn, Germany}
\email{schippa@math.uni-bonn.de}

\begin{abstract}
Decoupling inequalities for the parabola are obtained for weights in physical space. The decoupling iteration relies on a fractal bilinear Kakeya estimate. We record applications to exponential sums.
\end{abstract}

\maketitle

\section{Introduction}

\subsection{Decoupling and fractal local smoothing}
Originally conceived by Wolff \cite{Wolff2000} to make progress on Sogge's local smoothing conjecture for the wave equation \cite{Sogge1991}, decoupling inequalities have become an invaluable tool in Fourier restriction theory. Bourgain--Demeter \cite{BourgainDemeter2015} obtained sharp decoupling inequalities for the paraboloid $\mathbb{P}^d = \{ (\xi, |\xi|^2) \in \R^{d+1} : \xi \in \R^d, |\xi| \leq 1 \}$. These read at the Tomas-Stein exponent $p_d = \frac{2(d+2)}{d}$ for functions with Fourier support in an $R^{-1}$-neighborhood of the paraboloid $\text{supp}(\hat{F}) \subseteq \mathcal{N}_{R^{-1}}(\mathbb{P}^d)$:
\begin{equation}
\label{eq:DecouplingParaboloid}
\| F \|_{L^{p_d}(\R^{d+1})} \lesssim_\varepsilon R^{\varepsilon} \big( \sum_{\theta: \, R^{-\frac{1}{2}}-\text{ball}} \| P_{\theta} F \|^2_{L^{p_d}(\R^{d+1})} \big)^{\frac{1}{2}}.
\end{equation}
Above the sum is over finitely overlapping $R^{-\frac{1}{2}}$-cubes $\theta$ which cover the unit ball $B_d(0,1)$ and $P_{\theta}$ denotes the Fourier projection to $\theta \times \R$. The sharp decoupling inequalities are obtained for the cone as a corollary via the Pramanik--Seeger \cite{PramanikSeeger2007} argument (see also \cite{GarrigosSeeger2009}) and local smoothing estimates for the wave equation with sharp loss of derivatives follow in a corresponding integrability range.

Another consequence of \eqref{eq:DecouplingParaboloid} is the family of sharp periodic Strichartz estimates:
\begin{equation*}
\| e^{it \Delta} P_N f \|_{L^{p_d}([0,1] \times \T^d)} \lesssim_\varepsilon N^{\varepsilon} \| f \|_{L^2(\T^d)}
\end{equation*}
with $P_N$ denoting the projection to frequencies of size $N$.
Pioneering work on periodic Strichartz estimates is due to Zygmund \cite{Zygmund1974}. Bourgain systematically studied periodic Strichartz estimates \cite{Bourgain1993A,Bourgain1993B} to show low-regularity well-posedness for nonlinear dispersive equations on periodic domains. Our results yield periodic Strichartz estimates on fractal sets of times.

Here we consider decoupling estimates for $\text{supp}(\hat{F}) \subseteq \mathcal{N}_{2^{-r}}(\mathbb{P}^1)$, $R=2^r$:
\begin{equation*}
\| F \|_{L^6(w_{\mu})} \lesssim_\varepsilon R^{s(w_{\mu}) + \varepsilon} \big( \sum_{I \in \mathbb{I}_{r/2}} \| P_{I} F \|_{L^6(w_{\mu}^I)}^2 \big)^{\frac{1}{2}}
\end{equation*}
with $w_{\mu}$, $w_{\mu}^I \geq 0$ measurable functions, which resemble product weights $\R \times \bar{I}^{R}_{\mu}$ - on a ``fractal set of times" given by $\bar{I}^R_{\mu}$. For $r \geq 0$, $\mathbb{I}_{r}$ denotes a covering of $[-1,1]$ with finitely overlapping intervals of length $2^{-r}$ and for $I \subseteq \R$, $P_I$ the Fourier projection to $I \times \R$. We consider $L^p$-norms weighted with a measurable function $w \geq 0$. For $p \in [1,\infty)$ we let
\begin{equation*}
\| F \|^p_{L^p(w)} = \int |F|^p(x) w(x) dx, \quad F: \R^d \to \C \text{ measurable}.
\end{equation*}
The modification for $p=\infty$ is as usual.

There is a large body of literature on refined decoupling and Strichartz inequalities, which have in particular come up in the context of the proof of the Schr\"odinger maximal function estimate and the Falconer distance set problem \cite{DuGuthLiZhang2018, DuZhang2019, DuGuthOuWangWilsonZhang2021,DuOuRenZhang2023,DuOuRenZhang2023B}.  The refined Strichartz estimates do not rely on a genuine fractal Kakeya estimate, but on the multilinear Kakeya estimate itself, which becomes effective when considering a uniform distribution on a fractal output set. Notably, smoothing estimates with fractal measure are already mentioned in Wolff's pioneering work \cite[Section~6]{Wolff2000}; see \cite{BeltranRoosRutarSeeger2025,LeeRoncalZhangZhao2025,RoosSeeger2023} for recent contributions. The present work was motivated by the fractal local smoothing conjecture, which was brought up again in the recent survey of Roos--Seeger \cite{RoosSeeger2025}. 

In a recent preprint Demeter--Fu \cite{DemeterFu2026} obtain an improved decoupling estimate for functions with Fourier support in Ahlfors-regular subsets of the neighborhood of the parabola. In \cite{DemeterFu2026} the decoupling iteration is supplied with an improved fractal Kakeya estimate (formulated as a bilinear restriction estimate in \cite[Theorem~1.1]{DemeterFu2026}). However, this is a spectral assumption and not an estimate with weights in physical space.
In another direction, Kim--Ko \cite{KimKo2025} obtained weighted square function estimates - more precisely, weighted wave envelope estimates, in $L^4$. Again a bilinear restriction estimate is crucially used. We contrast our work with \cite{KimKo2025} further below. Another instance of quantified decoupling iterations is given in \cite{Schippa2025Strichartz}, which obtain multilinear Strichartz estimates on frequency-dependent time intervals on tori.

\subsection{Weighted decoupling via fractal Kakeya estimates}

In this paper, we leverage the fractal bilinear Kakeya estimate through the classical multi-scale analysis due to Bourgain--Demeter \cite{BourgainDemeter2015} to obtain a decoupling estimate in $L^6$. First, we comment on product weights which are the motivating example. Let $\bar{I}^{R}_{\mu} \subseteq \R$ be a measurable set. For an interval $J_M \subseteq \R$ of length $1 \leq M \leq R$ we let 
\begin{equation}
\label{eq:BorelCondition}
\mu(J_M) = \lambda^1(J_M \cap \bar{I}^R_{\mu})
\end{equation}
denote the one-dimensional Lebesgue measure of the intersection with $\bar{I}^R_{\mu}$. 
We require the doubling condition for mild dilations $J_{2M} \supseteq J_M$:
\begin{equation}
\label{eq:MildContinuity}
D^{-1} \mu(J_{2M}) \leq \mu(J_M) \leq \mu(J_{2M})
\end{equation}
for a constant $D > 0$.

Consider the weight $w(x_1,x_2) = \chi_{\bar{I}^R_{\mu}}(x_2)$ and transversely intersecting tubes of size $M^{\frac{1}{2}} \times M$ with long side making an angle with the $x_1$-axis: $\angle(\nu(T_i),e_1) \geq 1/8$. We let $w(A)= \int_A w(x) dx$ for $A \subseteq \R^2$ measurable.

In the fractal Kakeya estimate the quantity
\begin{equation}
\label{eq:kappa}
\kappa_M = \sup_{J_{M^{1/2}} \subseteq J_M: \text{ intervals} } \frac{M^{1/2} \mu(J_{M^{1/2}})}{\mu(J_M)}
\end{equation}
arises. Let $R \geq 1$. Already in the proof of weighted square function estimates parabolically rescaled quantities have to be considered. For $\nu \in [R^{-1/2},1] \cap 2^{\Z}$, let $\mu_{\nu^2}(J_M) = \mu(\nu^{-2} J_M)$. Correspondingly, we let for $1 \leq M \leq \nu R^{\frac{1}{2}}$ 
\begin{equation}
\label{eq:kappaLambda}
\kappa_M^{\nu} = \sup_{J_{M^{\frac{1}{2}}} \subseteq J_M: \text{ intervals} } \frac{M^{1/2} \mu_{\nu^2}(J_{M^{1/2}})}{\mu_{\nu^2}(J_M)}.
\end{equation}
For $A \geq 2$ let $[A]_d = \{ A^{\frac{1}{2^s}} : s \geq 0
 \} \cap [2,\infty) $ denote the double-dyadic intermediate scales between $2$ and $A$ and for $1 \leq A < 2$ let $[A]_d = \emptyset$. The double-dyadic scales appear naturally in the classical decoupling iteration.

Let $\bar{\kappa}^\nu = \prod_{M \in [\nu R^{1/2}]_d} \kappa_M^{\nu}$ for $[\nu R^{\frac{1}{2}}]_d \neq \emptyset$ and $\bar{\kappa}^{\nu} = 1$, else. We consider weights given as convex combinations
\begin{equation*}
\tilde{w}_{\mu}^{r,\nu,J} = \sum_{k \geq 0} a_k(\nu,J) w(\cdot + v_k(\nu,J)).
\end{equation*}
These weights can be made explicit in the square function iteration through parabolic rescaling and Galilean transformation. Since the applications use only the convex-combination structure, this is presently omitted.
For a measurable set $S$, $\chi_S$ denotes its indicator function. 

Our first result is a local $L^4$-square function estimate for product weights.
\begin{theorem}
\label{thm:FractalSquareFunction}
Let $R=2^r \geq 2$, $\bar{I}^{\mu}_R \subseteq \R$ be a measurable set such that $\mu$ defined as in \eqref{eq:BorelCondition} satisfies \eqref{eq:MildContinuity}. Set $w = \chi_{\R} \times \chi_{\bar{I}^R_{\mu}} \geq 0$.

 For $\text{supp}(\hat{F}) \subseteq \mathcal{N}_{R^{-1}}(\mathbb{P}^1)$ the following holds for $c \geq 1$ independent of $R$ and $\bar{I}^{R}_{\mu}$:
\begin{equation}
\| F \|^4_{L^4(w )} \lesssim \log(R)^c \sum_{R^{-1/2} \leq \nu \leq 1} \bar{\kappa}^{\nu} \int \big( \sum_{J \in \mathbb{I}_{r/2}} |P_J F|^2 (\tilde{w}_{\mu}^{r,\nu,J})^{\frac{1}{2}} \big)^{2}  dx .
\end{equation}
\end{theorem}
Further comparison with \cite{KimKo2025} is in order. Notably, the estimate in \cite{KimKo2025} is in terms of the usual (unweighted) square function:
\begin{equation*}
\| \big( \sum_{J \in \mathbb{I}_{r/2}} |P_J F|^2 \big)^{\frac{1}{2}} \|_{L^4},
\end{equation*}
whereas presently the weight is essentially retained. This facilitates iterating the expression in a fractal decoupling iteration. The present induction-on-scales clarifies that the quantity \eqref{eq:kappa} for a bilinear square function estimate has only to be considered on local scales $M \leq R^{\frac{1}{2}}$ (and rescalings) and not on the global scale $M=R$. Moreover, considering product weights we can lower the sub-polynomial loss in $R$ in \cite{KimKo2025} to a polylogarithmic factor and the dependence on $\bar{\kappa}$ at multiple scales becomes transparent through a Whitney decomposition. Let $\alpha \in (0,1]$. The expressions simplify for weights satisfying the following uniformity condition for $\nu \in [2^{-\frac{r}{2}},1]$ and $M \leq \nu 2^{\frac{r}{2}}$:
\begin{equation}
\label{eq:LocalAhlfors}
\sup_{J_{M^{\frac{1}{2}}} \subseteq J_M : \text{intervals}} \frac{M^{\frac{1}{2}} \mu_{\nu^2}(J_{M^{\frac{1}{2}}})}{\mu_{\nu^2}(J_{M})} \lesssim M^{\frac{1-\alpha}{2}}.
\end{equation}
This is reminiscent of Ahlfors $\alpha$-regularity, though strictly speaking, this would require a discretization to pass from a set of measure zero to a set with non-vanishing measure. And notably, the condition \eqref{eq:LocalAhlfors} is more localized.
Indeed, on the largest scale, which is the original spatial scale corresponding to $\nu =1$, we only have to require the $\alpha$-regularity condition on intervals up to length $R^{\frac{1}{2}}$. For smaller $\nu$, on the level of the original scale $R$, we have to consider intervals possibly larger than $R^{\frac{1}{2}}$, but the estimate can still be regarded as local because $M \ll R^{\frac{1}{2}}$ in this case.

As a consequence of \eqref{eq:LocalAhlfors}, $\bar{\kappa}^{\nu} \leq \log(R)^c (\nu^2 R)^{\frac{1-\alpha}{2}}$ from summing a geometric series, and we find under the assumptions of Theorem \ref{thm:FractalSquareFunction}:
\begin{equation}
\label{eq:SimplificationLocalAhlfors}
\| F \|_{L^4(w)} \lesssim \log(R)^{c} R^{\frac{1-\alpha}{8}} \sup_{\nu \in [2^{-r/2},1]} \bigg( \int \big( \sum_{J \in \mathbb{I}_{r/2}} |P_J F|^2 (\tilde{w}_{\mu}^{r,\nu,J})^{\frac{1}{2}} \big)^2  dx \bigg)^{\frac{1}{4}}.
\end{equation}

We obtain $L^4$-Strichartz estimates on the torus as a corollary, invoking Bourgain's approximation argument \cite{Bourgain2013} (see \cite[Section~2.2]{Schippa2025Strichartz} for a detailed exposition). For $\mathcal{I} \subseteq [0,1]$ let $\bar{I}^{N^2}_{\mu} = \bigcup_{k \in \Z} (N^2 \mathcal{I} + k N^2) $ be the periodization of the $N^2$-dilation.
\begin{theorem}
\label{thm:FractalL4Strichartz}
Let $N \in 2^{\N}$, $\mathcal{I} \subseteq [0,1]$ measurable with $\bar{I}_{\mu}^{N^2}$ satisfying \eqref{eq:MildContinuity} and \eqref{eq:LocalAhlfors} with $R=N^2$. Suppose that $|\mathcal{I}| \lesssim N^{2(\alpha' - 1)}$ for $\alpha' \leq \alpha$. Then the following estimate holds:
\begin{equation*}
\| e^{it \Delta} P_N f \|_{L^4_{t,x}(\mathcal{I} \times \T)} \lesssim \log(N)^c N^{\frac{2 \alpha'-\alpha-1}{4}  } \| f \|_{L^2(\T)}.
\end{equation*}
\end{theorem}

We turn to the main result: \emph{weighted decoupling inequalities} via fractal Kakeya estimates. 
For decoupling we consider the following slightly more general weights:
\begin{assumption}
\label{assump:weights}
There is $D_1 \geq 1$ such that the following holds. Let $M \geq 1$, and let $\mu$ be a Borel measure on $\R$, which satisfies \eqref{eq:MildContinuity}. We consider weights $w \geq 0$ which are measurable functions, with
\begin{equation}
\label{eq:AlmostProductWeights}
D_1^{-1} w(Q_M) \leq M \cdot \mu(J_M) \leq D_1 w(Q_M).
\end{equation}
\end{assumption}

As a consequence, we obtain for tubes $T_i$ with the dimensions and inclination specified above
\begin{equation*}
w(T_i) = \int_{\R^2} \chi_{T_i}(x) w(x) dx \leq \sum_{Q_{M^{\frac{1}{2}}} \in \mathcal{Q}_{M^{\frac{1}{2}}}(T_i)} \int_{Q_{M^{\frac{1}{2}}}} w(x) dx \sim M^{\frac{1}{2}} \mu(J_M)
\end{equation*}
for $J_M \subseteq \R$ an interval with $\pi_2(T_i) \subseteq J_M$ and note that $\lambda^1(\pi_2(T_i)) \sim M$. Above $\mathcal{Q}_{M^{\frac{1}{2}}}(T_i)$ denotes a finitely overlapping cover of $T_i$ with cubes of size $M^{\frac{1}{2}}$.
We introduce more notation: For $F \geq 0$ a measurable function, $A \subseteq \R^2$ a measurable set, we let
\begin{equation*}
\int_{A_{\mu}} F(x) dx = \int_A F(x) w(x) dx, \quad \dashint_{A_{\mu}} F(x) dx = \frac{1}{w(A)} \int_{A_{\mu}} F(x) dx.
\end{equation*}
In principle, when proving weighted decoupling, the constants $\kappa_M^{\lambda}$ seem tractable like in Theorem \ref{thm:FractalSquareFunction}. To simplify matters, we assume \eqref{eq:LocalAhlfors}.

The following is the main result:
\begin{theorem}
\label{thm:FractalDecouplingTheorem}
Let $\text{supp}(\hat{F}) \subseteq \mathcal{N}_{R^{-1}}(\mathbb{P}^1)$ and $w$ a weight satisfying Assumption \ref{assump:weights} and \eqref{eq:LocalAhlfors}. Then the following holds for any $\varepsilon > 0$:
\begin{equation}
\| F \|_{L^6(w)} \lesssim_{\varepsilon} R^{ \frac{1-\alpha}{6} + \varepsilon} \big( \sum_{I \in \mathbb{I}_{r/2}} \| P_I F \|^2_{L^6(w^{I,F})} \big)^{\frac{1}{2}}
\end{equation}
for $w^{I,F} = \sum_{k \in \N} a_k w(\cdot + v_k)$, $v_k \in \R^2$  convex combinations of translates of $w$.
\end{theorem}
We comment more on the weights in Section \ref{section:FractalDecoupling}. We allow the coefficients and translations to depend on $I$ and $F$. This yields the following periodic Strichartz estimate on a fractal set of times:
\begin{theorem}
\label{thm:FractalL6Strichartz}
Let $N \in 2^{\N}$, $\mathcal{I} \subseteq [0,1]$ with $\bar{I}^{N^2}_{\mu}$ satisfying \eqref{eq:MildContinuity} and \eqref{eq:LocalAhlfors}. Suppose that $|\mathcal{I}| \lesssim N^{2(\alpha' - 1)}$ for $\alpha' < \alpha$. Then the following estimate holds:
\begin{equation*}
\| e^{it \Delta} P_N f \|_{L^6_{t,x}(\mathcal{I} \times \T)} \lesssim_\varepsilon N^{(\alpha' - \alpha)/3 + \varepsilon} \| f \|_{L^2(\T)}.
\end{equation*}
\end{theorem}
The estimate $\|e^{it \Delta} P_N f \|_{L^6_{t,x}(\mathcal{I} \times \T)} \lesssim_\varepsilon N^{\varepsilon} \| f \|_{L^2(\T)}$ holds regardless of sparsity by the usual Strichartz estimates. To observe a gain for fractal sets of times, the condition $\alpha'<\alpha$ is vital, that is to say, a two-scale condition which separates a sparser global scale ($\alpha'$) from a denser local scale (given by $\alpha$). 


\subsection{Remarks}
It seems natural to extend the analysis in several ways in the future:
\begin{enumerate}
\item The decoupling estimates and applications are recorded under a local Ahlfors regularity assumption \eqref{eq:LocalAhlfors}. However, the analysis seems explicit enough to detect multiple scales, as illustrated in Theorem \ref{thm:FractalSquareFunction}. It remains to examine to what extent the Assouad spectrum is captured and quantified by the Legendre--Assouad function, which was used to formulate the fractal local smoothing conjecture for the wave equation \cite{BeltranRoosRutarSeeger2025,RoosSeeger2023,RoosSeeger2025}.
\item It is plausible that the strategy extends to higher dimensions as well. The cornerstone would be a fractal multilinear Kakeya estimate \cite{BennettCarberyTao2006,Guth2010,Guth2015}.
\end{enumerate}

\emph{Outline of the paper.} After collecting preliminaries and proving the fractal Kakeya estimate in Section \ref{section:Preliminaries}, we turn to the proof of the fractal square function estimate in Section \ref{section:SquareFunction}. The proof is an explicit induction-on-scales with ball inflation resting on the fractal Kakeya estimate. Then, in Section \ref{section:FractalDecoupling} we prove the fractal decoupling estimate for the parabola via a partial decoupling, which again rests on the fractal bilinear Kakeya, and induction-on-scales akin to the classical decoupling iteration. 
In Section \ref{section:ExpSum} we sketch the argument to prove the corresponding exponential sum estimates.

\section{Preliminaries}
\label{section:Preliminaries}
Before we turn to the proofs of the decoupling results, we record the key fractal Kakeya estimate, wave packet decompositions adapted to our setting, and the locally constant property.

\subsection{Fractal Kakeya estimate}

Let $1 \leq N \leq R$. Below we let $T_i \in \T_i$ denote tubes of dimensions $N^{1/2} \times N$, $\nu(T)$ its long side, and $\nu^{\perp}(T)$ its short side and $\angle(\nu(T_i),e_1) \geq 1/8$. We quantify their ``fractal" overlap in a fractal cube of size $N$.
Let $\bar{I}_{\mu}^N$ be like above satisfying \eqref{eq:MildContinuity} and $I_x^N$, $I^N_t$ intervals of length $N \geq 1$. Let $w$ be a weight satisfying Assumption \ref{assump:weights}.
\begin{proposition}
\label{prop:FractalKakeya}
Let $Q^N = I^N_x \times I^N_t$. Let 
\begin{equation*}
g_1 = \sum_{T_1 \in \T_1} a_{T_1} 1_{T_1}, \quad g_2 = \sum_{T_2 \in \mathbb{T}_2} b_{T_2} 1_{T_2}
\end{equation*}
with $T_i \cap Q^N \neq \emptyset$ for any $T_i \in \T_i$, $|\nu(T_1) \wedge \nu(T_2)| \geq \frac{1}{8}$ for any $T_i \in \T_i$ and $a_{T_1}, b_{T_2} > 0$. Then,
\begin{equation}
\label{eq:FractalKakeya}
\dashint_{Q_{\mu}^N} g_1 g_2 dx \lesssim \big( \sup_{J_{N^{\frac{1}{2}}} \subseteq I_t^N} \frac{N^{\frac{1}{2}} \mu(J_{N^{1/2}})}{\mu(I_t^N)} \big) \dashint_{3 Q_{\mu}^N} g_1 dx \dashint_{3 Q_{\mu}^N} g_2 dx
\end{equation}
with $3Q^N = 3 I^N_x \times 3 I^N_t $ denoting the mildly enlarged cube.
\end{proposition}
\begin{proof}
We find for any tubes $T_i \in \mathbb{T}_i$ by the transversality condition that $\pi_2(T_1 \cap T_2 \cap Q_N) \subseteq J_{C N^{\frac{1}{2}}}^{T_1 \cap T_2}$ for an interval $J_{C N^{\frac{1}{2}}}^{T_1 \cap T_2} \subseteq I_t^N$ of length $C N^{\frac{1}{2}}$. $C$ is a fixed constant depending on the quantitative transversality. Consequently,
\begin{equation}
\label{eq:IntersectionEstimate}
\dashint_{Q_{\mu}^N} \chi_{T_1} \chi_{T_2} \lesssim \frac{1}{\mu(I_t^N) N} N^{\frac{1}{2}} \mu(J_{C N^{\frac{1}{2}}}^{T_1 \cap T_2}) \lesssim \frac{1}{N} \sup_{J_{N^{\frac{1}{2}}} \subseteq I_t^N} \frac{N^{\frac{1}{2}} \mu(J_{N^{\frac{1}{2}}})}{\mu(I_t^N)}.
\end{equation}
To proceed, we note that for a tube $T_i \cap Q^N \neq \emptyset$ it holds $T_i \subseteq 3 Q^N$. So, for any $T_i \in \mathbb{T}_i$, we find for an interval $J_N^{T_i} \subseteq 3 I_t^N$:
\begin{equation*}
\int_{3 Q^N_{\mu}} T_i \sim N^{\frac{1}{2}} \mu(J_N^{T_i}).
\end{equation*}
By \eqref{eq:MildContinuity} we find $\mu(J_N^{T_i}) \sim \mu(I_t^N)$. Consequently,
\begin{equation}
\label{eq:IntersectionEstII}
\dashint_{3 Q^N_{\mu}} T_i \sim \frac{N^{\frac{1}{2}} \mu(J_N^{T_i})}{N \mu (3 J_N)} \sim N^{-\frac{1}{2}}.
\end{equation}
Taking \eqref{eq:IntersectionEstimate} and \eqref{eq:IntersectionEstII} together, \eqref{eq:FractalKakeya} follows from expanding $g_1 g_2$.
\end{proof}

\subsection{Fourier transform and wave packet decompositions}

The Fourier transform of $f \in \mathcal{S}(\R^d)$ is given by
\begin{equation*}
\hat{f}(\xi) = \int_{\R^d} e^{-i x \xi} f(x) dx
\end{equation*}
and the inverse Fourier transform of $g \in \mathcal{S}(\R^d)$ given by
\begin{equation*}
\check{g}(x) = \frac{1}{(2 \pi)^d} \int_{\R^d} e^{i x \xi} g(\xi) d\xi.
\end{equation*}
The extensions to $\mathcal{S}'$ are as usual.

We turn to the wave packet decomposition, which will allow us to decompose $P_{J_i} F$, and as a consequence $|P_{J_i} F|^2$ into tubes from above. A reference is \cite[Theorem~2.2]{Demeter2020}. For $F \in L^2$, $\text{supp}(\hat{F}) \subseteq \theta$, $\theta$ a rectangle of size $(2 \pi) \cdot ( 2^{-m} \times 2^{-2m})$ into directions $e_{\theta}$, $n_{\theta}$ with $\xi_{\theta}$ at its center choose $\tilde{\chi}_{\theta} \in C^{\infty}_c(\Theta)$, $\tilde{\chi}_{\theta} \equiv 1$ on $\theta$ with $\Theta$ of size $(2 \pi C) \cdot ( 2^{-m} \times 2^{-2m})$ a mild dilation of $\theta$. The normalized exponentials
\begin{equation*}
E_k(\xi)= |\Theta|^{-1/2} e^{ i x_k \cdot (\xi - \xi_{\theta})}, \quad k = (k_1,k_2) \in \Z^2,
\end{equation*}
form an orthonormal basis of $L^2(\Theta)$, where
\begin{equation*}
x_k = (2^m k_1 e_{\theta} + 2^{2m} k_2 n_{\theta})/C.
\end{equation*}
The Fourier transform of an (essentially) normalized packet is given by
\begin{equation*}
\hat{\phi}_k(\xi) = \tilde{\chi}_{\theta}(\xi) E_k(\xi)
\end{equation*}
and the coefficients are correspondingly given by
\begin{equation*}
\alpha_k = \langle F, \phi_k \rangle = \langle \hat{F}, \tilde{\chi}_{\theta} E_k \rangle.
\end{equation*}

In summary,
\begin{proposition}
Under the above assumptions, it holds
\begin{equation}
\label{eq:ExpansionF}
F = \sum_{k \in \Z^2} \langle F, \phi_k \rangle \phi_k = \sum_{k \in \Z^2} \alpha_k \phi_k.
\end{equation}
Moreover,
\begin{equation}
\label{eq:L2NormExpansion}
\| F \|_2^2  = \sum_{k \in \Z^2} |\alpha_k|^2
\end{equation}
and we have the pointwise bound for any $N \geq 1$:
\begin{equation}
\label{eq:PointwiseBound}
|\phi_k(x)| \leq C_N 2^{-\frac{3m}{2}} (1 + d_k(x))^{-N}, \; d_k(x) = 2^{-m} |(x-x_k) \cdot e_{\theta}| + 2^{-2m} |(x-x_k) \cdot n_{\theta}|.
\end{equation}
\end{proposition}
\begin{proof}
\eqref{eq:ExpansionF} is immediate in Fourier space by the completeness of $E_k$. More precisely,
\begin{equation*}
\sum_{k \in \Z^2} \alpha_k \hat{\phi}_k = \tilde{\chi}_{\theta} \sum_{k \in \Z^2} \langle \hat{F}, \tilde{\chi}_{\theta} E_k \rangle E_k = \tilde{\chi}_{\theta}^2 \hat{F} = \hat{F} \Rightarrow F = \sum_{k \in \Z^2} \alpha_k \phi_k.
\end{equation*}

Since $\tilde{\chi}_{\theta} \equiv 1$ on $\text{supp}(\hat{F}) \subseteq \theta$, Parseval's theorem yields by \eqref{eq:ExpansionF}
\begin{equation*}
\sum_{k \in \Z^2} |\alpha_k|^2 = \| \tilde{\chi}_{\theta} \hat{F} \|_{L^2}^2 = \| \hat{F} \|_2^2 = \| F \|_2^2.
\end{equation*}
By non-stationary phase estimates, we find that $\phi_k$ is essentially supported on a tube of dimensions $2^m \times 2^{2m}$ into directions $e_{\theta}$, $n_{\theta}$, which is centered at $x_k$, and $\phi_k$ is $L^2$-normalized, which is \eqref{eq:PointwiseBound}.
\end{proof}

We shall moreover label $\phi_k$ as $\phi_T$ for tubes $T \in \mathbb{T}_{\theta}$ of size $2^m \times 2^{2m}$, which tile $\R^2$ and are aligned dual to $\theta$.
 $T$ indicates the essential support of $\phi_T$ centered at $x_k$, which allows us to write:
\begin{equation*}
F = \sum_{T \in \T_{\theta}} \alpha_T \phi_T.
\end{equation*}
Localization properties of wave packet decompositions are moreover discussed in \cite[Section~2]{Demeter2020}.
Below is a precise statement used in this paper, which in turn has several analogs in the literature e.g. \cite[Lemma~3.4]{Guth2018}. 
\begin{lemma}
\label{lem:WavePacketLocalization}
Let $B_R \subseteq \R^2$ be a ball of radius $2^{2m} = R \geq 1$, $\text{supp}(\hat{F}) \subseteq \theta$, $F = \sum_{T \in \T_{\theta}} \alpha_T \phi_T$ like above. Then it holds for any $M \geq 1$
\begin{equation*}
\sum_{\substack{T \in \T_\theta: \\ T \cap B_R \neq \emptyset}} |\alpha_T|^2 \lesssim_M \| F \|_{L^2(w^M_{B_R})}^2
\end{equation*}
for weights $w^M_{B_R} = (1 + R^{-1} \text{dist}(x,B_R))^{-M}$ with high polynomial decay off $B_R$.
\end{lemma}
\begin{proof}
To prove the localized almost orthogonality, we carry out the conjugation
\begin{equation*}
G = (w_{B_R}^M)^{1/2} F, \quad \psi_T =(w_{B_R}^M)^{-1/2} \phi_T \text{ for } T \cap B_R \neq \emptyset.
\end{equation*}
Clearly,
\begin{equation*}
\alpha_T = \langle (w_{B_R}^M)^{1/2} F, (w_{B_R}^M)^{-1/2} \phi_T \rangle = \langle G, \psi_T \rangle.
\end{equation*}
Consider $\mathcal{T}:\ell^2 \to L^2$, $\mathcal{T} c= \sum_{\substack{T \in \T_{\theta}: \\ T \cap B_R \neq \emptyset}} c_T \psi_T$. We record
\begin{equation*}
|\psi_T(x)| \lesssim_{N,M} |T|^{-1/2} (1+d_T(x))^{-N+M}, \quad \sup_T \sum_{T'} |\langle \psi_T, \psi_{T'} \rangle| \lesssim 1.
\end{equation*}
This follows from the tubes intersecting $B_R$ and decaying much more rapidly away from $B_R$ than the weight. By Young's inequality we find the boundedness of $\mathcal{T}$. Since the norm of the adjoint coincides with the norm of the operator, we find
\begin{equation*}
\langle \mathcal{T} c, G \rangle = \sum_T c_T \langle \psi_T , G \rangle, \quad (\mathcal{T}^* G)_m = \langle \psi_T, G \rangle.
\end{equation*}
We find as a consequence of the boundedness of $\mathcal{T}^*$:
\begin{equation*}
\sum_T |\langle \psi_T, G \rangle|^2 \lesssim_M \| G \|_2^2 = \int |F|^2 w^M_{B_R}. \qedhere
\end{equation*}

\end{proof}

\subsection{The locally constant property}

It is a widely accepted heuristic that functions with Fourier support in a convex set $\theta$ are ``locally constant" on the dual set (cf. \cite{TaoVargas2000,BourgainGuth2011}). We make this rigorous by a \emph{scattered modulation sum}, which is obtained from the Fourier series expansion. Here we introduce notation similar to \cite{Lee2018} (which is slightly changed for our setup).

Let $d \in \N$, $\sigma > 0$, and large constants $M \geq 200 d$, $C_M \geq 1$. For $\xi_0 \in \R^d$ let $\mathfrak{q}(\xi_0,\sigma)$ be the cube with side length $\sigma$, centered at $\xi_0$. For $l \in  \Z^d$ let
\begin{equation*}
A_l = C_M (1+|l|)^{-M}, \quad \tau^{\sigma}_l f(x) = f(x-\sigma^{-1} l).
\end{equation*}
For $\sigma >0$, we define the scattered modulation sums $[F]_{\sigma}$, $[[F]]_{\sigma}$ by
\begin{equation*}
[F]_{\sigma}(x) = \sum_{l \in \Z^d} A_l |\tau^{\sigma}_l F(x)|, \quad [[F]]_{\sigma}(x) = \sum_{l_1,l_2 \in \Z^d} A_{l_1} A_{l_2} |\tau^{\sigma}_{l_1+l_2} F(x)|.
\end{equation*}
S. Lee recorded the following by Fourier series expansion:
\begin{lemma}[{\cite[Lemma~2.13]{Lee2018}}]
\label{lem:ScatteredModulationSum}
Let $\xi_0,x_0 \in \R^d$. Suppose that $F$ is a function with $\text{supp}(\hat{F}) \subseteq \mathfrak{q}(\xi_0,\sigma)$. Then, if $x \in \mathfrak{q}(x_0,1/\sigma)$, we find
\begin{equation*}
|F(x)| \leq [F]_{\sigma}(x_0) \leq [[F]]_{\sigma}(x).
\end{equation*}
\end{lemma}

\subsection{Local $L^2$-orthogonality}

We recall local $L^2$-orthogonality, cf. \cite[Prop.~10.20]{Demeter2020} and note that the degree of the power weight is not decreased when invoking it.
\begin{lemma}
\label{lem:LocalL2Orthogonality}
Let $M=M(d)\geq 10d$ and $w_R(x) = \big( 1 + \frac{|x-x_0|}{R} \big)^{-M}$. Let $(F_{\theta})_{\theta \in \mathcal{F}} \subseteq L^2(\R^d)$, $\text{supp}(\hat{F}_{\theta}) \subseteq \theta \subseteq \R^d$ with $(\theta + B_d(0,c R^{-1}))_{\theta \in \mathcal{F}}$ finitely overlapping. Then
\begin{equation*}
\| \sum_{\theta \in \mathcal{F}} F_{\theta} \|^2_{L^2(w_R)} \lesssim_{M,d} \sum_{\theta \in \mathcal{F}} \| F_{\theta} \|^2_{L^2(w_R)}.
\end{equation*}
\end{lemma}

\section{The fractal square function estimate}
\label{section:SquareFunction}
This section is devoted to the proof of Theorem \ref{thm:FractalSquareFunction}. The key point is to obtain a bilinear square function estimate through iterated application of the fractal bilinear Kakeya estimate:
\begin{equation*}
\big( \int_{Q_{\mu}^r} |F_1 F_2|^2 \big) \lesssim \log(R)^c \prod_{M \in [R^{\frac{1}{2}}]_d} \kappa_M  \int \prod_{i=1}^2 \big( \sum_{J_i \in \mathbb{I}_{r/2}(I_i)} |P_{J_i} F_i|^2 \big) \tilde{w}_{Q_{\mu}^r}
\end{equation*}
The assumptions are $\text{supp}(\hat{F}_i) \subseteq (I_i \times \R) \cap \mathcal{N}_{R^{-1}}(\mathbb{P}^1)$ and $|I_1| \sim |I_2| \sim \text{dist}(I_1, I_2) \sim 1$, and we let $R = 2^r \geq 2$. $c \geq 1$ is a constant uniform in the remaining parameters. Moreover,
\begin{equation*}
\tilde{w}_{Q^r_{\mu}} = \sum_{Q^{r/2} \in \mathcal{Q}^{r/2,r}} w_{Q^r_{\mu}}
\end{equation*}
is a weight based on $w$ in $Q^r$, which is rapidly decaying off $Q^r$, with more notation introduced below. To reduce the proof of Theorem \ref{thm:FractalSquareFunction} to the above, a standard linear-to-bilinear reduction via Whitney decomposition and almost orthogonality is carried out.

\subsection{The bilinear fractal square function iteration}

Here the fractal bilinear Kakeya estimate crucially comes into play.
Let $N \in 2^{\N}$, $r \geq 1$, $R=2^r$, and $w \geq 0$ a product weight like the assumptions. 

 Let $F: \R^2 \to \C$, $\text{supp}(\hat{F}) \subseteq \mathcal{N}_{\delta}(\mathbb{P}^1)$. Let $I_1$, $I_2 \subseteq [-2,2]$ be two intervals with $|I_1| \sim |I_2| \sim \text{dist}(I_1,I_2) \sim 1$. Given $m,r \geq 1$, $t \geq 2$, and $F: \R^2 \to \C$, we let
\begin{equation*}
D_t(m,Q^r_{\mu},F) = \big[ \big( \sum_{I \in \mathbb{I}_m(I_1)} \| P_I F \|^2_{L^t_{\sharp}(w_{Q^r_{\mu}})} \big) \big( \sum_{I \in \mathbb{I}_m(I_2)} \| P_I F \|^2_{L^t_{\sharp}(w_{Q^r_{\mu}})} \big) \big]^{\frac{1}{4}}.
\end{equation*}
$w_{Q^r_{\mu}}$ denotes a weight which is rapidly decaying off $Q^r_{\mu}$. The weight is defined as
\begin{equation*}
w_{Q^r_{\mu}} = \sum_{k \in \Z^2} (1+|k|)^{-100} (w \chi_{Q^r})(x + 2^{r} k) .
\end{equation*}
This weight is still local in the sense that it only depends on the properties on $w$ on $Q^r$, but it takes into account contributions off $Q^r$ as well. More technically speaking, $w_{Q^r_{\mu}}$ interacts well with the scattered modulation sum.

Let $\mathcal{Q}^r = \{ Q^r \} $ consist of the single cube $Q^r$.
For scales $m \in \{ \frac{r}{2}, \, \frac{r}{2^2}, \, \ldots \frac{r}{2^s} \}$, $s \in \N_0$ with $s$ chosen maximally such that $\frac{r}{2^s} \geq 2$, we consider families $\mathcal{Q}^{m}(Q^{2m})$ of finitely overlapping cubes $Q^m \subseteq Q^{2m}$, which are covering $Q^{2m} \in \mathcal{Q}^{2m}$. Next, we let for $m < r$, $\mathcal{Q}^{m,r} = \bigcup_{Q^{2m} \in \mathcal{Q}^{2m}} \mathcal{Q}^m(Q^{2m})$, that are the cubes of size $2^m$ which cover $Q^r$. The number of overlapping cubes is now bounded by $C\log(r/m)$, as every intermediate scale yields finite overlap. Second, we let for $p \in [4,6]$
\begin{equation*}
A_p^p(m,Q^r_{\mu},F) = \frac{1}{w(Q^r)} \sum_{Q^m_{\mu} \in \mathcal{Q}^r_{\mu}} w(Q^m) D^p_2(m,Q^m_{\mu},F).
\end{equation*}

In the following we suppose that $\delta \leq 2^{-r}$ and the expression is regarded as completely decoupled when $2^{-m} \sim \delta^{1/2}$. As we fix $F$ in the following, this will be suppressed in notation for $A_p$ and $D_p$. Leaning on terminology from classical decoupling theory, the following is referred to as fractal Kakeya leap:
\begin{proposition}[Fractal~Kakeya~leap]
\label{prop:BilinearSquareFunction}
With notation like above, $M=2^m$, and $r \geq 2m$, we have
\begin{equation*}
A_4^4(m,Q^r_{\mu},F) \lesssim \kappa_{M^2} A_4^4(2m,Q^r_{\mu},F).
\end{equation*}
\end{proposition}
We clarify that in Section \ref{section:FractalDecoupling} a more general argument proves Proposition \ref{prop:PartialDecoupling}, which recovers the above. The argument below is more direct and included for comparison reasons.

\subsubsection{An outline of the argument}

We first provide a sketch of the argument, which we will make rigorous subsequently. It turns out that by elementary arguments, it suffices to consider $r=2m$.

\emph{Wave packet decomposition.} As a heuristic, we suppose that the weight $w_{Q^{m}_{\mu}}$ coincides with $w \chi_{Q^m}$. Let $x \in Q^r$. Consider the wave packet decomposition 
\begin{equation*}
P_{J_i} F = \sum_{T_i \in \T_{\theta_{J_i}}} \alpha_{T_i} \phi_{T_i}(x)
\end{equation*}
with $\theta_{J_i}$ the $2^{-m} \times 2^{-2m}$-rectangle containing the Fourier support of $\text{supp}(\widehat{P_{J_i} F})$. For brevity let $\T_{J_i} = \T_{\theta_{J_i}}$. As another heuristic, we discard the tubes not intersecting with $Q^r$. We let
\begin{equation*}
P_{J_i} F \approx \sum_{T_i \in \T'_{J_i}} \alpha_{T_i} \phi_{T_i}(x),
\end{equation*}
in which expansion we consider only $T_i \in \T'_{J_i}$ for which $T_i \cap Q^r \neq \emptyset$. We collect the effectively contributing tubes in $\T' = \bigcup_{J_i} \T_{J_i}'$. For the contributing wave packets, we apply Jensen's inequality to find
\begin{equation}
\label{eq:JensenHeuristic}
\begin{split}
|P_{J_i} F|^2(x) &\leq \big| \sum_{T_i \in \T_{J_i}'} |\alpha_{T_i} \phi_{T_i}(x)| \big|^2 = \big| \sum_{T_i \in \T_{J_i}'} |\alpha_{T_i} |\phi_{T_i}(x)| 2^{\frac{3m}{2}} \big|^2 2^{-3m} \\
&\lesssim \sum_{T_i \in \T_{J_i}'} |\alpha_{T_i}|^2 |\phi_{T_i}(x)| 2^{\frac{-3m}{2}} = \sum_{T_i \in \T_{J_i}'} a_{T_i} \tilde{\phi}_{T_i}(x).
\end{split}
\end{equation}
Above we let $a_{T_i} = |\alpha_{T_i}|^2$ and $\tilde{\phi}_{T}(x) = 2^{-\frac{3m}{2}} |\phi_{T}(x)|$ for brevity. Clearly, $\tilde{\phi}_T$ has the same local constant properties and is $L^1$-normalized. By Lemma \ref{lem:WavePacketLocalization} we have
\begin{equation*}
\sum_{T_i \in \T_{J_i}'} a_{T_i} \lesssim \sum_{T_i : T_i \cap Q^r \neq \emptyset} |\alpha_{T_i}|^2 \lesssim_K \| P_{J_i} F \|^2_{L^2(w^K_{Q^r})}.
\end{equation*}

\emph{The locally constant property.} By the above we find
\begin{equation*}
\begin{split}
&\quad \big( \sum_{J_1 \in \mathbb{I}_m(I_1)} \| P_{J_1} F \|^2_{L^2_{\sharp}(Q^m_{\mu})} \big) \big( \sum_{J_2 \in \mathbb{I}_m(I_2)} \| P_{J_2} F \|^2_{L^2_{\sharp}(Q^m_{\mu})} \big) \\
&\lesssim \dashint_{Q^m_{\mu}} \sum_{T_1} a_{T_1} \tilde{\phi}_{T_1} \dashint_{Q^m_{\mu}} \sum_{T_2} b_{T_2} \tilde{\phi}_{T_2}.
\end{split}
\end{equation*}

By the Fourier support properties of $\phi_T$, we have that heuristically $\tilde{\phi}_T$ is essentially constant on $Q^m$. For this reason
\begin{equation*}
\dashint_{Q^m_{\mu}} \sum_{T_1 \in \T_1'} a_{T_1} \tilde{\phi}_{T_1} \dashint_{Q^m_{\mu}} \sum_{T_2 \in \T_2'} b_{T_2} \tilde{\phi}_{T_2} \sim \dashint_{Q^m_{\mu}} \sum_{T_i \in \T_i'} a_{T_1} b_{T_2} \tilde{\phi}_{T_1} \tilde{\phi}_{T_2},
\end{equation*}
which suggests
\begin{equation}
\label{eq:HeuristicA4}
A_4^4(m,Q^{2m}_{\mu},F) \approx \dashint_{Q^r_{\mu}} \big( \sum_{T_1 \in \T_1'} a_{T_1} \tilde{\phi}_{T_1} \big) \big( \sum_{T_2 \in \T_2'} b_{T_2} \tilde{\phi}_{T_2} \big).
\end{equation}

\emph{The fractal bilinear Kakeya estimate.} To bring the fractal bilinear Kakeya estimate to bear, we use the locally constant property of $\tilde{\phi}_{T_1}$. Tile $\R^2$ with copies of $T_i$ and for $\tilde{T}_i$ a $(2^m k_1,2^{2m} k_2)$-translate (into short/long side) of $T_i$, and we let $\text{dist}_{par}(\tilde{T}_i,T_i) = |k_1|+|k_2|$. More precisely, we estimate
\begin{equation}
\label{eq:PhiEstimate}
\tilde{\phi}_{T_i} \lesssim_K 2^{-3m} \sum_{\tilde{T}_i \parallel T_i} 1_{\tilde{T}_i}(x) (1+\text{dist}_{par}(\tilde{T}_i,T_i))^{-K}.
\end{equation}

Consequently, by changing the summation, we find for $x \in Q^r$
\begin{equation*}
\begin{split}
\sum_{T_1} a_{T_1} \tilde{\phi}_{T_1}(x) &\lesssim_K \sum_{T_1} 2^{-3m} a_{T_1} \sum_{\tilde{T}_1 \parallel T_1} 1_{\tilde{T}_1}(x) (1+\text{dist}_{par}(\tilde{T}_1,T_1))^{-K} \\
&\lesssim \sum_{\tilde{T}_1} 2^{-3m} 1_{\tilde{T}_1}(x) \sum_{T_1 \parallel \tilde{T}_1} a_{T_1} (1+\text{dist}_{par}(\tilde{T}_1,T_1))^{-K} \\
&\lesssim \sum_{\tilde{T}_1} 2^{-3m} 1_{\tilde{T}_1}(x) \tilde{a}_{\tilde{T}_1}, \quad \tilde{a}_{\tilde{T}_1} := \sum_{T_1 \parallel \tilde{T}_1} a_{T_1} (1+\text{dist}_{par}(\tilde{T}_1,T_1))^{-K}.
\end{split}
\end{equation*}
The resulting expression is amenable to the fractal bilinear Kakeya inequality because the tubes are transversely interacting by the Fourier support condition. Let
\begin{equation*}
g_1(x) = 2^{-3m} \sum_{T_i: T_i \cap Q^r \neq \emptyset} \tilde{a}_{T_i} 1_{T_i}, \quad g_2(x) = 2^{-3m} \sum_{T_i : T_i \cap Q^r \neq \emptyset} \tilde{b}_{T_i} 1_{T_i}, \quad x \in Q^r.
\end{equation*}
We apply Proposition \ref{prop:FractalKakeya} to find
\begin{equation}
\label{eq:BilinearKakeya}
\dashint_{Q^r_{\mu}} g_1 g_2 dx \lesssim \kappa_{M^2} \dashint_{3 Q^r_{\mu}} g_1 \dashint_{3 Q^r_{\mu}} g_2.
\end{equation}

We integrate and use the convolution structure of the coefficients to find
\begin{equation*}
\begin{split}
\dashint_{3 Q^r_{\mu}} \sum_{T_1} \tilde{a}_{T_1} 1_{T_1} &\lesssim 2^{-4 m} \sum_{T_1} \tilde{a}_{T_1} \\
 &\lesssim 2^{-4m} \sum_{T_1: T_1 \cap Q^r \neq \emptyset} a_{T_1} \lesssim_K \frac{1}{|Q^r|} \| P_{I_1} F \|_{L^2(w^K_{Q^r})}^2,
\end{split}
\end{equation*}
where in the end we used Lemma \ref{lem:WavePacketLocalization}.

The remainder of the argument is close to the usual decoupling. We shall be brief.
Note that above a weight is permissible with arbitrarily high polynomial decay off $Q^r$. 
We use local $L^2$-orthogonality, recalled in Lemma \ref{lem:LocalL2Orthogonality}, to find
\begin{equation*}
\| P_{I_1} F \|^2_{L^2(w^K_{Q^r})} \lesssim \sum_{J_1 \in \mathbb{I}_{2m}(I_1)} \| P_{J_1} F \|^2_{L^2(w^K_{Q^r})}.
\end{equation*}
Finally, we can use the locally constant property of $P_{J_1} F$ on $Q^r$ to argue that
\begin{equation*}
\frac{1}{|Q^r|} \| P_{J_1} F \|^2_{L^2(w^K_{Q^r})} \lesssim \frac{1}{w(Q^r)} \|P_{J_1} F \|^2_{L^2(w_{Q^r_{\mu}})}.
\end{equation*}
Here we use that the initial weight on $Q^r$ with decay $K$ is more rapidly decaying than the fixed polynomial weight $w_{Q^r_{\mu}}$ - as we may.

\subsubsection{The proof of the fractal bilinear square function estimate}

We begin the proof in earnest.
\begin{proof}[Proof~of~Proposition~\ref{prop:BilinearSquareFunction}]
For $J_1 \in \mathbb{I}_m(I_1)$ let $\theta_{J_1}$ be a rectangle of size $2^{-m} \times 2^{-2m}$ such that $\text{supp}(\widehat{P_{J_1} F}) \subseteq \theta_{J_1}$. We plug in the wave packet decomposition
\begin{equation*}
P_{J_1} F(x) = \sum_{T_1 \in \T_{J_1} } \alpha_{T_1} \phi_{T_1}(x).
\end{equation*}
Above $\T_{J_1}$ denotes the tiling of $\R^2$, $T_1 \parallel \theta_{J_1}^*$. Next, we use Jensen's inequality like in \eqref{eq:JensenHeuristic} to find the pointwise inequality with $a_{T_1} = |\alpha_{T_1}|^2$:
\begin{equation*}
\sum_{J_1 \in \mathbb{I}_m(I_1)} |P_{J_1} F|^2(x) \lesssim 2^{-\frac{3m}{2}} \sum_{T_1 \in \T_1} a_{T_1} |\phi_{T_1}(x)|.
\end{equation*}
Here we take into consideration \emph{all} wave packets and not just the ones intersecting $Q^r$. By keeping track of the rapid decay of wave packets off $Q^r$, we can as well estimate the contributions of wave packets, which are separated from $Q^r$.

Next, we write out the fractal weight:
\begin{equation*}
\begin{split}
\frac{1}{w(Q^m)} \int \sum_{J_1 \in \mathbb{I}_m(I_1)} |P_{J_1} F|^2 w_{Q^m_{\mu}}(x) dx &\lesssim \sum_{T_1 \in \mathbb{T}_1} a_{T_1} 2^{-\frac{3m}{2}} \\
&\quad \quad  \sum_{k \in \Z^2} (1+|k|)^{-100} \dashint_{Q^m_{\mu} } |\phi_{T_1}(y+2^m k)| dy,
\end{split}
\end{equation*}
and similarly, for the sum over $J_2 \in \mathbb{I}_m(I_2)$, denoting coefficients with $b_{T_2}$.

Let $\sigma = 2^{-m}$. Then, as a consequence of the scattered modulation sum estimate in Lemma \ref{lem:ScatteredModulationSum}, we find for $x \in Q^m$:
\begin{equation*}
\begin{split}
&\quad \sum_{T_1 \in \mathbb{T}_1} a_{T_1} 2^{-\frac{3m}{2}} \sum_{k \in \Z^2} (1+|k|)^{-100} \dashint_{Q^m_\mu } |\phi_{T_1}(y+2^m k)| dy \\
&\lesssim \sum_{T_1 \in \mathbb{T}_1} a_{T_1} 2^{-\frac{3m}{2}} \sum_{k \in \Z^2} (1+|k|)^{-100} [\phi_{T_1}]_{\sigma}(x+2^m k).
\end{split}
\end{equation*}
Define functions $G_i$, $i=1,2$ via
\begin{equation*}
\begin{split}
G_1(x) = \sum_{T_1 \in \T_1} a_{T_1} 2^{-\frac{3m}{2}} \sum_{k \in \Z^2} (1+|k|)^{-100} [\phi_{T_1}]_{\sigma}(x+2^m k), \\
G_2(x) = \sum_{T_2 \in \T_2} b_{T_2} 2^{-\frac{3m}{2}} \sum_{k \in \Z^2} (1+|k|)^{-100} [\phi_{T_2}]_{\sigma}(x+2^m k).
\end{split}
\end{equation*}
Clearly, for $x \in Q^m$ we find
\begin{equation*}
\sum_{J_i \in \mathbb{I}_m(I_i)} \| P_{J_i} F \|^2_{L_{\sharp}^2(w_{Q^m_{\mu}})} \lesssim G_i(x).
\end{equation*}
For this reason we obtain by taking a weighted averaged:
\begin{equation*}
\dashint_{Q^m_{\mu}} \sum_{J_1 \in \mathbb{I}_m(I_1)} \|P_{J_1} F \|^2_{L^2_{\sharp}(w_{Q^m_{\mu}})} \sum_{J_2 \in \mathbb{I}_m(I_2)} \|P_{J_2} F \|^2_{L^2_{\sharp}(w_{Q^m_{\mu}})} \lesssim \dashint_{Q^m_{\mu}} G_1(x) G_2(x) dx.
\end{equation*}

Before we proceed, we remark on the properties of
\begin{equation*}
2^{-\frac{3m}{2}} \sum_{k \in \Z^2} (1+|k|)^{-100} [\phi_{T_1}]_{\sigma}(x+2^m k) = \tilde{\phi}_{T_1}(x),
\end{equation*}
which involves two averages on the scale $2^m$. The first via the scattered modulation sum can be taken to be of arbitrarily high power, so effectively it does not change the decay and locally constant properties of $\phi_{T_1}$. For $[\phi_{T_1}]_{\sigma}$ we use the same estimate \eqref{eq:PhiEstimate}. The second average is taken with respect to $w_{Q^m}$ decaying with high (but not arbitrarily high) polynomial degree - again on scale $2^m$.
We shall see that this averaging is permissible for our purpose because for a tube and $Q^r$ to be separated on a scale $A 2^r$, $A \geq 16$, this induces a decay factor of $A^{-100} 2^{-97m}$. 

In the following we let $r=2m$. The general case will follow from elementary means, which is like in the classical decoupling argument and detailed at the end of the proof of Proposition \ref{prop:PartialDecoupling}.
So, summing over $Q^m \in \mathcal{Q}^{m,r}$ we find by finite overlap
\begin{equation*}
\frac{1}{w(Q^r)} \sum_{Q^m \in \mathcal{Q}^{m,r}} w(Q^m) \dashint_{Q^m_{\mu}} G_1(x) G_2(x) dx \lesssim \dashint_{Q^r_{\mu}} G_1(x) G_2(x) dx.
\end{equation*}
In the heuristic from above, we restricted to $T_1 \in \mathbb{T}_1: T_1 \cap Q^r \neq \emptyset$. To make the localization rigorous, we carry out a dyadic decomposition.  We consider the tubes $T_i$ with $\text{dist}(T_i,Q^r) \lesssim 2^r$ and $\text{dist}(T_i,Q^r) \sim A 2^r$ for $A \geq 16$, $A \in 2^{\N}$.
Precisely, we break up
\begin{equation*}
\begin{split}
G_i(x) &=  \sum_{T_i} a_{T_i} \tilde{\phi}_{T_i}(x) \\
&=  \sum_{T_i \cap 16 Q^r \neq \emptyset } a_{T_i} \tilde{\phi}_{T_i}(x) + \sum_{\substack{A \geq 16, \\ A \in 2^{\N}}} \sum_{\text{dist}(T_i,Q^r) \sim A 2^r } a_{T_i} \tilde{\phi}_{T_i} \\
&=: G_{i0}(x) + \sum_{A \geq 16} G_{iA}(x).
\end{split}
\end{equation*}

The contribution of $G_{i0}$ comprised of the tubes $\mathbb{T}_{i0} = \{ T_i \in \mathbb{T}_i : \text{dist}(T_i,Q^r) \leq 16 \cdot 2^r \}$ can be estimated like in the heuristic.
We find for $T_1$ similar to the above for $x \in Q^r$:
\begin{equation*}
\tilde{\phi}_{T_1}(x) \lesssim 2^{-3m} \sum_{\tilde{T}_1 \parallel T_1} 1_{\tilde{T}_1}(x) (1+\text{dist}_{par}(\tilde{T}_1,T_1))^{-98}.
\end{equation*}
The average into the direction of the short side with fixed high polynomial degree decreases the arbitrarily high polynomial decay, but it is still sufficient for the purpose of the argument. Indeed, we can proceed like above via
\begin{equation*}
\begin{split}
\sum_{T_1 \in \mathbb{T}_{10}} a_{T_1} \tilde{\phi}_{T_1}(x) &\lesssim \sum_{\tilde{T}_1} 1_{\tilde{T}_1}(x) 2^{-3m} \tilde{a}_{\tilde{T}_1} =: g_1(x), \\ \tilde{a}_{\tilde{T}_1} &= \sum_{\substack{T_1 \parallel \tilde{T}_1, \\ T_1 \in \mathbb{T}_{10}}} a_{T_1} (1+\text{dist}_{par}(\tilde{T}_1,T_1))^{-50}
\end{split}
\end{equation*}
and corresponding definition of $g_2$. Then, we can use the fractal bilinear Kakeya estimate, the equivalence between the $\ell^1$ norm of $a_T$ and $\tilde{a}_T$ to find
\begin{equation}
\label{eq:RigidFractalKakeya}
\begin{split}
&\quad \dashint_{Q^r_{\mu}} g_1(x) g_2(x) dx \\ &\lesssim \kappa_{M^2} 2^{-4m} \sum_{T_1: T_1 \cap 16 Q^r \neq \emptyset} a_{T_1} 2^{-4m} \sum_{T_2: T_2 \cap 16 Q^r \neq \emptyset} b_{T_2} \\
&\lesssim \kappa_{M^2} 2^{-4m} \sum_{T_1: T_1 \cap 16 Q^r \neq \emptyset} |\alpha_{T_1}|^2 \cdot 2^{-4m} \sum_{T_2 : T_2 \cap 16 Q^r \neq \emptyset} |\beta_{T_2}|^2.
\end{split}
\end{equation}
Next, we can apply Lemma \ref{lem:WavePacketLocalization} to continue
\begin{equation*}
\eqref{eq:RigidFractalKakeya} \lesssim \kappa_{M^2} \prod_{i=1}^2 \sum_{J_i \in \mathbb{I}_m(I_i)} \| P_{J_i} F \|^2_{L_{\sharp}^2(w^K_{Q^r})}.
\end{equation*}
Here we can choose $K=300$, which leads to a weight which is decaying much faster off $Q^r$ than the corresponding $w_{Q^r_{\mu}}$. Then, we can apply $L^2$-local orthogonality Lemma \ref{lem:LocalL2Orthogonality} to find
\begin{equation*}
\sum_{J_i \in \mathbb{I}_m(I_i)} \| P_{J_i} F \|^2_{L_{\sharp}^2(w^K_{Q^r})} \lesssim \sum_{J_i \in \mathbb{I}_{2m}(I_i)} \| P_{J_i} F \|^2_{L^2_{\sharp}(w^K_{Q^r})}.
\end{equation*}
By the locally constant property, it can be expected that we can replace $w_{Q^r}^K$ with its fractal counterpart. We employ again a scattered modulation sum. Let $\tilde{\sigma} = 2^{-r}$. We have by Lemma \ref{lem:ScatteredModulationSum}
\begin{equation*}
\dashint_{Q^r + k 2^r} |P_{J_i} F|^2(x) dx \leq [P_{J_i} F]_{\tilde{\sigma}}^2(x_k) \leq \dashint_{Q^r_\mu + k 2^r} [[P_{J_i} F]]^2_{\tilde{\sigma}}(x) dx
\end{equation*}
with $x_k = c(Q^r) + k 2^r$, $c(Q^r)$ denoting the center of the cube, and by Jensen's inequality
\begin{equation*}
[[P_{J_i} F]]^2_{\tilde{\sigma}}(x) \lesssim \sum_{l_1,l_2 \in \Z^2} A_{l_1} A_{l_2} |\tau^{\tilde{\sigma}}_{l_1+l_2} P_{J_i} F|^2(x).
\end{equation*}
Hence, choosing the weight $A_l$ sufficiently large, e.g. of size $K=300$, we find
\begin{equation*}
\sum_{k \in \Z^2} (1+|k|)^{-300} \sum_{l_1,l_2 \in \Z^2} A_{l_1} A_{l_2} \dashint_{Q^r_{\mu}+ 2^r(k+l_1+l_2)} |P_{J_i} F|^2 \lesssim \| P_{J_i} F \|^2_{L^2_{\sharp}(w_{Q^r_{\mu}})}
\end{equation*}
because we can absorb the higher weights into the weight $w_{Q^r_{\mu}}$ .

\medskip

Next, we estimate the contribution of the remainder terms involving $G_{iC}$ for $C \geq 16$.
The isotropic averaging of $\phi_{T_1}$ for $\text{dist}(T_1,Q^r) \sim A 2^r$ yields for $x \in Q^r$ by splitting the isotropic average into $|k| \ll A 2^m$ and $|k| \gtrsim A 2^m$:
\begin{equation*}
\tilde{\phi}_{T_1}(x) \lesssim 2^{-3m} \sum_{\tilde{T}_1 \parallel T_1} 1_{\tilde{T}_1}(x) \big[ (1+\text{dist}_{par}(\tilde{T}_1,T_1))^{-K} + A^{-100} 2^{-97m } \big].
\end{equation*}
We split 
\begin{equation*}
2^{3m} \tilde{\phi}_{T_1}(x) \lesssim A^{-K+3} \sum_{\tilde{T}_1 \parallel T_1} 1_{\tilde{T}_1}(x) (1+\text{dist}_{par}(\tilde{T}_1,T_1))^{-3} + A^{-100} 2^{-97m} \sum_{\tilde{T}_1 \parallel T_1} 1_{\tilde{T}_1}(x)
\end{equation*}
and correspondingly, $G_{1A} \lesssim G_{1A}^1 + G_{1A}^2$. Let
\begin{equation*}
\mathcal{E}_{A,r} = \{ x \in \R^d: \text{dist}(x,Q^r) \sim A 2^r \} \text{ and } w_{\mathcal{E}_{A,r}}^K = (1+2^{-r} \text{dist}(x,\mathcal{E}_{A,r}))^{-K}.
\end{equation*}

The first term is clearly permissible as this gives a weighted estimate for the wave packet coefficients and by Lemma \ref{lem:WavePacketLocalization} we find
\begin{equation*}
\begin{split}
&\quad \dashint_{Q^r_{\mu}} G_{1A}^1 G_{2B}^1 dx \\
 &\lesssim \kappa_{M^2} (2^{-4m})^2 A^{-K+3} \sum_{T_1: \text{dist}(T_1,Q^r) \sim A 2^r} |\alpha_{T_1}|^2 \cdot B^{-K+3} \sum_{T_2: \text{dist}(T_2,Q^r) \sim B 2^r} |\beta_{T_2}|^2 \\
&\lesssim \kappa_{M^2} 2^{-4m} A^{-K+3} \sum_{J_1 \in \mathbb{I}_m(I_1)} \| P_{J_1} F \|^2_{L^2(w^{K}_{\mathcal{E}_{A,r}})} 2^{-4m} B^{-K+3} \sum_{J_2 \in \mathbb{I}_m(I_2)} \| P_{J_2} F \|^2_{L^2(w^{K}_{\mathcal{E}_{B,r}})}.
\end{split}
\end{equation*}
Consequently,
\begin{equation*}
\sum_{A,B \geq 16} \dashint_{Q^r_{\mu}} G_{1A}^1 G_{2B}^1 dx \lesssim \kappa_{M^2} \prod_{i=1}^2 \sum_{J_i \in \mathbb{I}_m(I_i)} \| P_{J_i} F \|^2_{L^2_{\sharp}(w^{K-3}_{Q^r})},
\end{equation*}
which suffices.
The estimate for the second term can be carried out more directly and crossover terms containing products of $G_{i0}$ and $G_{iA}$ are estimated similarly. At this point, we can follow along the above lines, invoking local $L^2$-orthogonality by Lemma \ref{lem:LocalL2Orthogonality}, and then, recasting the standard weight as fractal weight using a scattered modulation sum.
\end{proof}

We record consequences of the above. The first is a straight-forward iteration:
\begin{corollary}
\label{cor:IteratedSquareFunctionIteration}
The following holds for $m,k,r \geq 0$, $2^k m \leq r/2$, $\text{supp}(\hat{F}) \subseteq \mathcal{N}_{2^{-r}}(\mathbb{P}^1)$:
\begin{equation*}
A_4(m,Q^r_{\mu},F)^4 \leq C^k \kappa_{2^{2m}} \kappa_{2^{4m}} \ldots \kappa_{2^{2^k m}} A_4(2^k m, Q^r_{\mu}, F)^4.
\end{equation*}
\end{corollary}

This yields the following bilinear square function estimate:
\begin{proposition}
\label{prop:BilinearSquareFunctionEstimate}
Let $I_1,I_2 \subseteq [-1,1]$ be two intervals with $|I_i| \geq 1/8$, which are separated $\text{dist}(I_1,I_2) \geq 1/8$ and $\text{supp}(\hat{F}) \subseteq \mathcal{N}_{2^{-r}}(\mathbb{P}^1)$ for $r \geq 10$. It holds
\begin{equation*}
\int_{Q^r_{\mu}} |P_{I_1} F|^2 |P_{I_2} F|^2 dx \leq r^c \big( \prod_{M \in [2^{r/2}]_{d}} \kappa_M \big) w(Q^r) A_4^4(r/2,Q^r_{\mu},F)
\end{equation*}
with $c$ uniform in $Q^r_{\mu}$, $F$, $I_j$, and $r$.
\end{proposition}
\begin{proof}
We divide the left-hand side by $w(Q^r)$ upon which we shall see it becomes comparable to $A_4(r/2^s,Q^r_{\mu},F)^4$ with $s$ chosen maximally such that $r/2^s \in [2,4)$. Indeed, we can break $I_i$ into finitely overlapping $I \in \mathbb{I}_{r/2^s}$ to find
\begin{equation*}
\begin{split}
&\quad \dashint_{Q^r_{\mu}} \sum_{J_1 \in \mathbb{I}_{r/2^s}(I_1)} |P_{J_1} F|^2 \sum_{J_2 \in \mathbb{I}_{r/2^s}(I_2)} |P_{J_2} F|^2 \\
&\leq \frac{1}{w(Q^r)} \sum_{Q^m \in \mathcal{Q}^{m,r}} w(Q^m) \dashint_{Q^m_{\mu}} \sum_{J_1 \in \mathbb{I}_{r/2^s}(I_1)} |P_{J_1} F|^2 \sum_{J_2 \in \mathbb{I}_{r/2^s}(I_2)} |P_{J_2} F|^2.
\end{split}
\end{equation*}

Next we use the locally constant property from Lemma \ref{lem:ScatteredModulationSum} to find for $\sigma =2^{-\frac{r}{2^s}}$
\begin{equation*}
\begin{split}
&\quad \dashint_{Q^m_{\mu}} \sum_{J_1 \in \mathbb{I}_{r/2^s}(I_1)} |P_{J_1} F|^2 \sum_{J_2 \in \mathbb{I}_{r/2^s}(I_2)} |P_{J_2} F|^2 \\
&\lesssim \dashint_{Q^m_{\mu}} \sum_{J_1 \in \mathbb{I}_{r/2^s}(I_1)} [[P_{J_1} F]]^2_{\sigma}(x) dx \dashint_{Q^m_{\mu}} \sum_{J_2 \in \mathbb{I}_{r/2^s}(I_2)} [[P_{J_2} F]]^2_{\sigma}(x) dx \\
&\lesssim  \sum_{J_1 \in \mathbb{I}_{r/2^s}(I_1)} \| P_{J_1} F \|^2_{L^2_{\sharp}(w_{Q^m_{\mu}})} \sum_{J_2 \in \mathbb{I}_{r/2^s}(I_2)} \| P_{J_2} F \|^2_{L^2_{\sharp}(w_{Q^m_{\mu}})}.
\end{split}
\end{equation*}
The argument to go from the iterated modulation sum to the weight $w_{Q^m_{\mu}}$ is like above and not detailed. Hence, by Corollary \ref{cor:IteratedSquareFunctionIteration} we find
\begin{equation}
\label{eq:IterationConsequence}
\dashint_{Q^r_{\mu}} |P_{I_1} F|^2 |P_{I_2} F|^2 dx \lesssim r^c \big( \prod_{M \in [2^{r/2}]_{d}} \kappa_M \big) A_4(r/2,Q^r_{\mu},F)^4.
\qedhere
\end{equation}
\end{proof}

\subsection{The linear-to-bilinear reduction}

We are ready for the proof of Theorem \ref{thm:FractalSquareFunction}, which will follow from the above bilinear estimate and a standard linear to bilinear reduction.

\begin{proof}[Proof~of~Theorem~\ref{thm:FractalSquareFunction}]
We reduce the linear estimate to the bilinear estimate by Whitney decomposition and almost orthogonality, cf. \cite{TaoVargasVega1998}. We write
\begin{equation*}
F^2 = \sum_{R^{-\frac{1}{2}} \leq \nu \leq 1} \sum_{\theta_1 \sim_{\nu} \theta_2} P_{\theta_1} F P_{\theta_2} F,
\end{equation*}
decomposing the Fourier support into intervals $\theta_1 \sim_{\nu} \theta_2$ of length $\nu$ at distance $\nu \in 2^{\Z}$; the multiplier of $P_{\theta}$ given by $\chi_{\theta} \times \chi_{\R}$. For $\nu \in [R^{-\frac{1}{2}}, 2 R^{-\frac{1}{2}})$ the intervals are no longer assumed to be separated, but contained in a joint interval of $4 \nu$.

For fixed $\nu$ we can use fixed-time almost orthogonality as follows. Then it is straight-forward that the Fourier transforms of $P_{\theta_1} F(t) P_{\theta_2} F(t)$ for $\theta_1 \sim_{\nu} \theta_2$ are supported in intervals of length $4 \nu$, essentially in a $\nu$-neighborhood of $\theta_1 + \theta_2$. Now, for two pairs $\theta_1 \sim_{\nu} \theta_2$ and $\theta_3 \sim_{\nu} \theta_4$ with $\text{dist}(\theta_1,\theta_3) \geq 10 \nu$, clearly $\text{dist}(\theta_1 + \theta_2, \theta_3 + \theta_4) \geq 5 \nu$. For this reason, by integrating over $\bar{I}^R_{\mu}$,
\begin{equation*}
\big\| \sum_{\theta_1 \sim_{\nu} \theta_2} P_{\theta_1} F P_{\theta_2} F \|^2_{L^2_{t,x}(\bar{I}^R_{\mu} \times \R)}  \lesssim \sum_{\theta_1 \sim_{\nu} \theta_2} \| P_{\theta_1} F P_{\theta_2} F \|^2_{L^2_{t,x}(\bar{I}^R_{\mu} \times \R)}.
\end{equation*}
For $\nu \in [1/32,1]$ we decompose $\R^2$ into cubes $Q^r$ of size $2^r$ we can apply Proposition \ref{prop:BilinearSquareFunctionEstimate} to find for $\theta_1 \sim_{\nu} \theta_2$
\begin{equation*}
\| P_{\theta_1} F P_{\theta_2} F \|^2_{L^2(Q^r_{\mu})} \lesssim r^c \big( \prod_{M \in [2^{r/2}]_d} \kappa_M \big) \int \prod_{i=1}^2 \big( \sum_{J_i \in \mathbb{I}_{r/2}(\theta_i)} |P_{J_i} F|^2 \big) \tilde{w}_{Q_{\mu}^r}.
\end{equation*}
Consequently, summing over $Q^r$ we find
\begin{equation*}
\| P_{\theta_1} F P_{\theta_2} F \|^2_{L^2(\bar{I}^R_{\mu} \times \R)} \lesssim r^c \big( \prod_{M \in [2^{r/2}]_d} \kappa_M \big) \int \prod_{i=1}^2 \big( \sum_{J_i \in \mathbb{I}_{r/2}(\theta_i)} |P_{J_i} F|^2 \big) \tilde{w}_{\mu}.
\end{equation*}

For $\nu \ll 1$ we apply parabolic rescaling $t \to \nu^2 t$, $x \to \nu (x + 2 \xi_0 \nu t)$:
\begin{equation*}
\big\| P_{\theta_1} F P_{\theta_2} F \|^2_{L^2(\bar{I}^R_{\mu})} = \big\| P_{\tilde{\theta}_1} \tilde{F} P_{\tilde{\theta}_2} \tilde{F} \big\|^2_{L^2_{t,x}(\bar{I}^{R \nu^2}_{\nu^2 \mu} \times \R )}.
\end{equation*}
After rescaling, the time domain is given by $\nu^2 \bar{I}^R_{\mu} = \bar{I}^{\nu^2 R}_{\nu^2 \mu}$. The rescaled intervals $\tilde{\theta}_i$ are approximately of unit size and separated on the unit scale. The Jacobian was absorbed into the rescaled function $\tilde{F}$, as we shall soon reverse the scaling.

Next, we partition the domain into cubes of size $\nu^2 2^r$:
\begin{equation*}
\big\| P_{\tilde{\theta}_1} \tilde{F} P_{\tilde{\theta}_2} \tilde{F} \big\|^2_{L^2_{t,x}(\bar{I}^{R \nu^2}_{\nu^2 \mu} \times \R )} = \sum_{Q^{r+ \log(\nu^2)}_{\nu^2 \mu}} \big\| P_{\tilde{\theta}_1} \tilde{F} P_{\tilde{\theta}_2} \tilde{F} \big\|^2_{L^2_{t,x}(Q^{r+ \log(\nu^2)}_{\nu^2 \mu})}.
\end{equation*}
For the further decomposed expression we can apply Proposition \ref{prop:BilinearSquareFunctionEstimate}:
\begin{equation*}
\begin{split}
\| P_{\tilde{\theta}_1} \tilde{F} P_{\tilde{\theta}_2} \tilde{F} \|^2_{L^2(Q^{r+ \log(\nu^2)}_{\mu})} &\lesssim (r \log(\nu^2))^c \big( \prod_{M \in [2^{r/2} \nu]_d} \kappa^{\nu}_M \big) \\
&\quad \times \int \prod_{i=1}^2 \big( \sum_{J_i \in \mathbb{I}_{r+\log_2(\nu^2))/2}(\tilde{\theta}_i)} |P_{J_i} \tilde{F} |^2 \big) \tilde{w}_{Q^{r + \log(\nu^2)}_{\nu^2 \mu}}.
\end{split}
\end{equation*}
Summing over the fractal cubes of size $2^r \nu^2$, we find
\begin{equation*}
\begin{split}
&\quad \sum_{Q^{r + \log(\nu^2)}_{\mu}} \| P_{\tilde{\theta}_1} \tilde{F} P_{\tilde{\theta}_2} \tilde{F} \|^2_{L^2(Q^{r + \log(\nu^2)}_{\mu})} \\
&\lesssim \log(R)^c \big( \prod_{M \in [2^{r/2} \nu]_d} \kappa^{\nu}_M \big) \int \prod_{i=1}^2 \big( \sum_{J_i \in \mathbb{I}_{(r+\log_2(\nu^2))/2}(\tilde{\theta}_i)} |P_{J_i} \tilde{F} |^2 \big) \tilde{w}_{\nu^2 \mu}^{r,\nu}.
\end{split}
\end{equation*}
$\tilde{w}_{\nu^2 \mu}^{r,\nu}$ is a convex combination of $\nu^2 \bar{I}^{R}_{\mu} \times \R$, consequently, upon reversing the scaling, we find a convex combination of $\bar{I}^R_{\mu} \times \R$ denoted by $\tilde{w}^{r,\nu}_{\mu}$.
Second, the intervals $I \in \mathbb{I}_{(r+\log(\nu^2))/2}$ are mapped to intervals of length $2^{-r/2}$, which are contained in an interval $\tilde{\theta} \supseteq \theta_1 \cup \theta_2$ of length $4 \nu$. We find
\begin{equation*}
\| P_{\theta_1} F P_{\theta_2} F \|^2_{L^2(w)} \lesssim \log(R)^c \bar{\kappa}^{\nu} \int \prod_{i=1}^2 \big( \sum_{J_i \in \mathbb{I}_{r/2}(\theta_i)} |P_{J_i} F|^2 \big) \tilde{w}^{r,\nu}_{\mu} 
\end{equation*}
with
\begin{equation*}
\bar{\kappa}^{\nu} = \prod_{M \in [2^{\frac{r}{2}} \nu ]_d} \kappa_M^{\nu}.
\end{equation*}
Consequently,
\begin{equation*}
\sum_{\theta_1 \sim_{\nu} \theta_2} \| P_{\theta_1} F P_{\theta_2} F \|^2_{L^2(w)} \lesssim \log(R)^c \bar{\kappa}^{\nu} \int \big( \sum_{J \in \mathbb{I}_{r/2}} |P_J F|^2 (\tilde{w}^{r,\nu,J}_{\mu})^{\frac{1}{2}} \big)^2
\end{equation*}
It remains to carry out the dyadic sum over $\nu$ via H\"older's inequality to find
\begin{equation*}
\| F \|^4_{L^4(w )} \lesssim \log(R)^c \sum_{R^{-1/2} \leq \nu \leq 1} \bar{\kappa}^{\nu} \int \big( \sum_{J \in \mathbb{I}_{r/2}} |P_J F|^2 (\tilde{w}_{\mu}^{r,\nu,J})^{\frac{1}{2}} \big)^{2}  dx . \qedhere
\end{equation*}
\end{proof}

\section{The fractal decoupling iteration}
\label{section:FractalDecoupling}
The outset is similar to the above: Let $r \geq 1$, $R := 2^r$. We consider a measurable set $\bar{I}^{r}_{\mu} \subseteq \R$ which satisfies \eqref{eq:MildContinuity} and \eqref{eq:LocalAhlfors}. Notation $Q^r$, $Q^r_{\mu}$, $D_t$ and $w_{Q^r_{\mu}}$ is like above. We aim to run the decoupling iteration from \cite{BourgainDemeter2015} whose bedrock is a partial decoupling for $A_6$. We have to pay attention to the fractal weights, which behave a little different from the standard rapidly decaying weight, and indeed, have to carry out the induction allowing for slightly more general weights. Iterating and bootstrapping the decoupling constant follows essentially the Bourgain--Demeter \cite{BourgainDemeter2015} argument. The presentation of the argument below follows Demeter's monograph \cite{Demeter2020}. We carry out the following steps, which are described in the next subsection:
\begin{enumerate}
\item Linear-to-bilinear reduction,
\item Jump-starting the decoupling iteration for the bilinear expression,
\item The partial decoupling via the fractal Kakeya estimate,
\item Iterating the partial decoupling,
\item Bootstrapping the linear decoupling constant through the decoupling iteration.
\end{enumerate}

\subsection{An overview of the argument}

It would be desirable to obtain a global decoupling estimate for $w$ like in Assumption \ref{assump:weights} and $\text{supp}(\hat{F}) \subseteq \mathcal{N}_{2^{-r}}(\mathbb{P}^1)$:
\begin{equation*}
\| F \|_{L^p(w)} \leq \text{Dec}(p,r,\alpha) \big( \sum_{I \in \mathbb{I}_{r/2}} \| P_I F \|^2_{L^p(w)} \big)^{\frac{1}{2}}.
\end{equation*}

The issue is that the decoupling argument requires to consider weighted translations of $w$. For this reason, we consider more general decoupling inequalities
\begin{equation*}
\| F \|_{L^p(w)} \leq C \big( \sum_{I \in \mathbb{I}_{r/2}} \| P_I F \|^2_{L^p(w^{I,F})} \big)^{\frac{1}{2}}
\end{equation*}
for admissible weights $w^{I,F} : \R^d \to \R_{\geq 0}$; see below. 
\begin{assumption}
Let $\alpha \in (0,1]$, $r \geq 1$. In the following we consider $\alpha$-admissible weights $w^r$ at scale $R=2^r$ satisfying Assumption \ref{assump:weights} with Borel measure $\mu$. We require that for $\nu \in [2^{-\frac{r}{2}},1]$ and $1 \leq M \leq \nu 2^{\frac{r}{2}}$, and intervals $J_{\nu^{-2} M^{\frac{1}{2}}} \subseteq J_{\nu^{-2} M}$ of length $\nu^{-2} M^{\frac{1}{2}}$, $\nu^{-2} M$, respectively, it holds
\begin{equation}
\label{eq:QuantifiedLocalAhlfors}
\sup_{J_{\nu^{-2} M^{\frac{1}{2}}} \subseteq J_{\nu^{-2} M}} \frac{M^{\frac{1}{2}} \mu(J_{\nu^{-2} M^{\frac{1}{2}}} )}{\mu(J_{\nu^{-2} M} )} \leq D_2 M^{\frac{1-\alpha}{2}}.
\end{equation}
\end{assumption}

Again we consider weights $w'$ which are convex combinations of translates of $w^r$, that is
\begin{equation*}
w' = \sum_{k \in \N} a_k w^r(\cdot + v_k)
\end{equation*}
with $a_k > 0$ for any $k \in \N$, $\sum_{k \geq 1} a_k =1 $, and $(v_k) \subseteq \R^2$.
The constants $D,D_1,D_2$ from \eqref{eq:MildContinuity}, \eqref{eq:AlmostProductWeights}, \eqref{eq:QuantifiedLocalAhlfors} are fixed. The induction is only carried out over weights with doubling and regularity constants bounded by $CD$ \eqref{eq:MildContinuity}, $CD_1$ \eqref{eq:AlmostProductWeights}, $D_2$ \eqref{eq:QuantifiedLocalAhlfors}. By choosing the $\alpha$-admissible weight $w^r$ at scale $2^r$, in the induction-on-scales only rescalings (and Galilean shearings) of $w^r$ arise.
In the motivating example $w^r(x_1,x_2) = \chi_{\bar{I}^R_{\mu}}(x_2)$ these are simple rescalings $w^{r_1}(x) = w^r(2^{r-r_1} x)$.

To justify the above, we clarify the stability of weights under parabolic rescaling. To this end, recall that parabolic rescaling is composed of the shearing $A(x,t) = (x+2 \xi_0 t, t)$ and parabolic dilation $D_{\nu}(x,t) = (\nu x, \nu^2 t)$. Let $T= D_{\nu} \circ A$ the transformation in physical space. We have, as we are considering unit frequencies, $|\xi_0| \leq 1$ and by linearity it is enough to check one rescaling. 
\begin{lemma}[Stability~lemma]
Let $\nu \in [2^{-\frac{r}{2}},1]$, and let $w^r$ be a weight satisfying Assumption \ref{assump:weights} with Borel measure $\mu$ with doubling constant $D$ \eqref{eq:MildContinuity}, and regularity constants $D_1,D_2$ \eqref{eq:AlmostProductWeights}, \eqref{eq:LocalAhlfors}. Then, $w^r_{\nu} = \nu w^r \circ T^{-1}$ satisfies Assumption \ref{assump:weights} at scale $\nu^2 2^r$ with Borel measure $\mu_{\nu^2} = \mu(\nu^{-2} \cdot)$ and doubling constant $D$, regularity constants $C D_1$ \eqref{eq:AlmostProductWeights}, $D_2$ \eqref{eq:LocalAhlfors} with $C \geq 1$ a uniform constant.
\end{lemma}
\begin{proof}
The doubling property is immediate. For the regularity properties we consider the linear transformation separately. First, let $\nu = 1$. $A^{-1}$ maps cubes $Q_M = I_M \times I_t^M$ to shearings, with time-interval unchanged. Since $|\xi_0| \leq 1$, clearly $A^{-1} Q_M \subseteq I'_{3M} \times I_t^M$ such that $w^r(A^{-1} Q_M) \leq 3 D_1 \mu(I_t^M)$. For the reverse bound cover $A^{-1} Q_M$ from the inside with cubes of size comparable to $M$. Clearly, since $|\xi_0| \leq 1$, with cubes of size $M/4$ we can cover the full height. Consequently, $w(A^{-1} Q_M) \geq D^{-1}_1 \frac{M}{4} \mu(J_M)$. So, $w^r \circ A^{-1}$ with $\mu$ satisfies Assumption \ref{assump:weights} with mildly enlarged regularity constants at scale $2^r$.

Next, we consider a weight $w^r$ and claim that $w^r_{\nu} = \nu w^r \circ D_{\nu}^{-1}$ satisfies Assumption \ref{assump:weights} at scale $\nu^2 2^r$ with Borel measure $\mu_{\nu^2}(A) = \mu(\nu^{-2} A)$ and regularity constants unchanged.

Indeed,
\begin{equation*}
\begin{split}
w_{\nu}^r(Q_M) &= \nu w^r(I_{\nu^{-1} M} \times I_{\nu^{-2} M}) = \nu \sum_{I_{\nu^{-1} M} \subseteq I_{\nu^{-2} M}} w^r(I_{\nu^{-1} M} \times I_{\nu^{-1} M}) \\
&\leq \nu D_1 \nu^{-1} M \sum_{I_{\nu^{-1} M}} \mu(I_{\nu^{-1} M}) \leq \nu D_1 \nu^{-1} M \sum_{I_{\nu^{-1} M}} \mu(I_{\nu^{-1} M}) \\
&\leq D_1 M \mu_{\nu^{-2}}(I_M).
\end{split}
\end{equation*}
The reverse bound is proved along the same lines.
\end{proof}

 The dependence on $D,D_1,D_2$ will be suppressed in the following. It is implicit that the decoupling estimate depends on $w^r$. 
The global decoupling constant $\text{Dec}(p,r,\alpha)$ is defined as infimum over $C> 0$ such that for any $\text{supp}(\hat{F}) \subseteq \mathcal{N}_{2^{-r}}(\mathbb{P}^1)$ like above, there are $w^{I,F}$, which are convex combinations of translates of $w^r$, such that
\begin{equation*}
\| F \|_{L^p(w^r)} \leq C \big( \sum_{I \in \mathbb{I}_{r/2}} \| P_I F \|^2_{L^p(w^{I,F})} \big)^{\frac{1}{2}}.
\end{equation*}
Let $I_i \subseteq [-1,1]$ and $\text{dist}(I_1,I_2) \geq 1/8$. $\text{BilDec}(p,r,\alpha)$ is defined as infimum over $C_2$ such that for $\text{supp}(\hat{F}) \subseteq \mathcal{N}_{2^{-r}}( \mathbb{P}^1)$ there are weights $w^{I,F}$ such that
\begin{equation*}
\| |P_{I_1} F P_{I_2} F|^{1/2} \|_{L^p(w^{r})} \leq C_2 \big( \sum_{I \in \mathbb{I}_{r/2}} \| P_I F \|^2_{L^p(w^{I,F})} \big)^{\frac{1}{2}}.
\end{equation*}

\subsection{Linear-to-bilinear reduction} Clearly, $\text{BilDec} \leq \text{Dec}$ by H\"older's inequality.
Like in the original decoupling iteration, we shall take advantage of a bilinear Kakeya estimate - here we use the fractal Kakeya estimate. To bring the bilinear estimate to bear, we reduce the linear to a bilinear estimate by a simple broad--narrow argument. We do not use a Whitney decomposition like in the previous section as this does not comply with a non-product weight. 

In the following we focus on the critical exponent $p=6$. We can estimate the linear decoupling constant in terms of the bilinear constant as follows:
\begin{lemma}
\label{lem:LinearBilinearReduction}
It holds for any $\delta > 0$:
\begin{equation*}
\text{Dec}(6,r,\alpha) \leq C_{\delta} 2^{r \delta} (1+ \sup_{1 \leq m \leq r} \text{BilDec}(6,m,\alpha)).
\end{equation*}
\end{lemma}
\begin{proof}
We find the pointwise bound by broad--narrow decomposition:
\begin{equation*}
|F(x)| \leq 4 \max_{I \in \mathbb{I}_k} |P_I F(x)| + 2^{c k} \sum_{\substack{J_1,J_2 \in \mathbb{I}_k : \\ 2 J_1 \cap 2 J_2 = \emptyset }} |P_{J_1} F(x) P_{J_2} F(x)|^{\frac{1}{2}}.
\end{equation*}
For the first term we find
\begin{equation*}
\| \max_{I \in \mathbb{I}_k} |P_I F| \|_{L^p(w)} \leq \big( \sum_{I \in \mathbb{I}_k} \| P_I F \|^p_{L^p(w)} \big)^{\frac{1}{p}}.
\end{equation*}
Fix $I \in \mathbb{I}_k$. By a Galilean transform $x \to x + 2 \xi_I t$, which in frequency space centers $I$ at the origin, and parabolic rescaling $x \to 2^k x$, $t \to 2^{2k} t$, we find
\begin{equation*}
\| P_I F \|_{L^p(w)} = \| \tilde{F} \|_{L^p(\tilde{w})}
\end{equation*}
with $\tilde{w}$ $\alpha$-admissible at scale $2^{r-2k}$ (and notably, constants $D,D_1,D_2$ from \eqref{eq:MildContinuity}, \eqref{eq:AlmostProductWeights}, \eqref{eq:QuantifiedLocalAhlfors} unchanged) and $\text{supp}(\widehat{\tilde{F}}) \subseteq \mathcal{N}_{2^{2k-r}}(\mathbb{P}^1)$. Consequently, by the definition of the decoupling constant and the embedding $\ell^2 \hookrightarrow \ell^p$ we find
\begin{equation*}
\| \max_{I \in \mathbb{I}_k} |P_I F| \|_{L^p(w)} \leq C \text{Dec}(6,r-2k,\alpha) \big( \sum_{I \in \mathbb{I}_{r/2}} \| P_I F \|_{L^p(w^{I,F})}^2 \big)^{\frac{1}{2}}.
\end{equation*}
For the bilinear expression with separated $\text{dist}(J_1,J_2) \geq 2^k$, we use again a parabolic rescaling to find
\begin{equation*}
\| |P_{J_1} F P_{J_2} F|^{\frac{1}{2}} \|_{L^p(w)} \lesssim \text{BilDec}(6,r-2m+1,\alpha) \big( \sum_{I \in \mathbb{I}_{r/2}} \| P_I F \|^2_{L^p(w^{I,F})} \big)^{\frac{1}{2}}.
\end{equation*}
Here $J_1 = [a,a+2^{-k}]$, $J_2 = [b,b+2^{-k}]$ and $2^{-m} \leq b-a < 2^{-m+1}$ for $m \in \N$. The details follow the classical case, for which we refer to \cite{Demeter2020}.

In summary (cf. \cite[Eq.~(10.7)]{Demeter2020}),
\begin{equation*}
\text{Dec}(6,r,\alpha) \leq C \text{Dec}(6,r-2k,\alpha) + 2^{c k} \sup_{1 \leq m \leq n} \text{BilDec}(6,m,\alpha).
\end{equation*}
Iterating the above yields
\begin{equation*}
\text{Dec}(6,r,\alpha) \lesssim C^{\frac{r}{2k}} (1 + 2^{ck} r \sup_{1 \leq m \leq r} \text{BilDec}(6,m,\alpha)).
\end{equation*}
Choosing $k = \delta r$ finishes the proof.
\end{proof}

It suffices to analyze the bilinear decoupling quantity. 
Let $\text{supp}(\hat{F_i}) \subseteq (I_i \times \R) \cap \mathcal{N}_{2^{-r}}(\mathbb{P}^1)$ with $|I_1| \sim |I_2| \sim \text{dist}(I_1,I_2) \sim 1$.
In the first step, to analyze the global quantity we decompose the plane into cubes of size $2^r$:
\begin{equation}
\label{eq:GlobalDecomposition}
\| |F_1 F_2|^{1/2} \|_{L^6(w)}^6 = \sum_{Q^r \in \mathcal{Q}^r(\R^2)} \| |F_1 F_2|^{1/2} \|_{L^6(Q^r)}^6.
\end{equation}
$\mathcal{Q}^r(\R^2)$ denotes a partition of $\R^2$ into cubes of size $2^r$. Most of the time it will be convenient to analyze a local averaged quantity:
\begin{equation*}
\| |F_1 F_2|^{1/2} \|_{L_{\sharp}^6(w)}^6 := \frac{1}{w(Q^r)} \| |F_1 F_2|^{1/2} \|_{L^6(Q^r_{\mu})}^6.
\end{equation*}
Having analyzed the local expression, we obtain a multi-scale inequality for the global decoupling inequality by virtue of \eqref{eq:GlobalDecomposition}. Recall that $2^r$ is the scale on which $w$ is $\alpha$-admissible.

\subsection{Jump-starting the decoupling iteration}
In the previous section we used Kakeya leaps essentially starting from the unit scale. Here we start at $m=(r/2)/2^s$ with $s$ determined below - starting from the unit scale will be too costly. We use the following which relies on the reverse H\"older inequality through the scattered modulation sum:
\begin{lemma}
Let $m= r/2^s$. The following holds:
\begin{equation*}
\begin{split}
&\quad \frac{1}{w(Q^r)} \| |F_1 F_2|^{1/2} \|^6_{L^6(Q^r_{\mu})} \\
&\leq 2^{6 r/2^s} \frac{1}{w(Q^r)} \sum_{Q^m \in \mathcal{Q}^{m,r}} w(Q^m) \big[ \sum_{J_1 \in \mathbb{I}_m(I_1)} \| P_{J_1} F \|^2_{L^2_{\sharp}(w_{Q^m_{\mu}})} \sum_{J_2 \in \mathbb{I}_m(I_2)} \| P_{J_2} F \|^2_{L^2_{\sharp}(w_{Q^m_{\mu}})} \big]^{\frac{3}{2}} \\
 &= 2^{\frac{6r}{2^s}} A_6^6(r/2^s,Q^r_{\mu},F).
\end{split}
\end{equation*}
\end{lemma}
\begin{proof}
First we decompose $Q^r$ into cubes $Q^m$ and then apply the Cauchy-Schwarz inequality:
\begin{equation*}
\begin{split}
&\quad \| |F_1 F_2|^{\frac{1}{2}} \|_{L^6(Q^r_{\mu})}^6 \\
 &= \sum_{Q^m \in \mathcal{Q}^{m,r}} \| |F_1 F_2|^{\frac{1}{2}} \|^6_{L^6(Q^m_{\mu})} \\
&\leq \sum_{Q^m \in \mathcal{Q}^{m,r}} \| F_1 \|^3_{L^6(Q^m_{\mu})} \| F_2 \|^3_{L^6(Q^m_{\mu})} \\
&\leq 2^{\frac{6 r}{2^s}} \sum_{Q^m \in \mathcal{Q}^{m,r}} \big( \sum_{J_1 \in \mathbb{I}_{r/2^s}(I_1)} \| P_{J_1} F_1 \|^2_{L^6(Q^m_{\mu})} \big)^{\frac{3}{2}} \big( \sum_{J_2 \in \mathbb{I}_{r/2^s}(I_2)} \| P_{J_2} F_2 \|^2_{L^6(Q^m_{\mu})} \big)^{\frac{3}{2}} \\
&= 2^{\frac{6 r}{2^s}} \sum_{Q^m \in \mathcal{Q}^{m,r}} w(Q^m_{\mu}) \big( \sum_{J_1 \in \mathbb{I}_{r/2^s}} \| P_{J_1} F_1 \|^2_{L^6_{\sharp}(Q^m_{\mu})} \big)^{\frac{3}{2}} \big( \sum_{J_2 \in \mathbb{I}_{r/2^s}} \| P_{J_2} F_2 \|^2_{L^6_{\sharp}(Q^m_{\mu})} \big)^{\frac{3}{2}}.
\end{split}
\end{equation*}
Now we estimate the averaged $L^6$-norm by an $L^2$-norm using the scattered modulation sum on the scale $\sigma = 2^{-m}$. Let $x_q \in Q^m_{\mu}$. Appealing to Lemma \ref{lem:ScatteredModulationSum}, we find
\begin{equation*}
\big( \dashint_{Q^m_{\mu}} |P_{J_1} F|^6 dx \big)^{\frac{1}{6}} \leq \big( \dashint_{Q^m_{\mu}} [P_{J_1} F]_{\sigma}^2(x_q) dx \big)^{\frac{1}{2}} \leq \big( \dashint_{Q^m_{\mu}} [[P_{J_1} F]]_{\sigma}^2(x) dx \big)^{\frac{1}{2}}.
\end{equation*}
Jensen's inequality yields
\begin{equation*}
[[P_{J} F]]_{\sigma}^2(x) \lesssim \sum_{l_1,l_2 \in \Z^2} A_{l_1} A_{l_2} |P_{J} F|^2(x+2^m (l_1 +l_2)).
\end{equation*}
We find
\begin{equation*}
\big( \int_{Q^m_{\mu}} [[P_{J_1} F]]_{\sigma}^2(x) dx \big)^{\frac{1}{2}} \lesssim \big( \int_{\R^2} |P_{J_1} F|^2(x) w_{Q^m_{\mu}}(x) dx \big)^{\frac{1}{2}},
\end{equation*}
which finishes the proof.
\end{proof}

\subsection{The partial decoupling} Proposition \ref{prop:BilinearSquareFunctionEstimate} is generalized by the following. We introduce the weights $w^{J}_{Q^{r}_{\mu}}$, which are combinations of $(w \chi_{3 Q^r})(\cdot + v)$ (see Proposition \ref{prop:FractalKakeya} for notation), and defined as follows: For $J \in \mathbb{I}_m$ let $\theta = (J \times \R) \cap \mathcal{N}_{2^{-r}}(\mathbb{P}^1)$ be a rectangle of size $2^{-m} \times 2^{-2m}$ and $\theta^*$ be the dual set, which is a rectangle of size $2^m \times 2^{2m}$ into directions $e_{\theta}$, $n_{\theta}$. We let
\begin{equation}
\label{eq:GeneralizedWeight}
w_{3 Q^r_{\mu}}^{J}(x) = \sum_{k,l \in \Z^2} (1+|k|)^{-100} (1+|l|)^{-200} (w \cdot \chi_{3 Q^r})(x + k 2^m + l_1 e_{\theta} 2^m + l_2 n_{\theta} 2^{2m})
\end{equation}
and
\begin{equation*}
\| F \|^{p}_{L^{p}_{\sharp}(w_{3 Q^r_{\mu}}^{J_i})} := \frac{1}{w(3 Q^r)} \| F \|_{L^p(w_{3 Q^r_{\mu}}^{J_i})}^p.
\end{equation*}

\begin{proposition}
\label{prop:PartialDecoupling}
Let notation be like above, $p \in [4,6]$, and $r/2 \geq 2m$. Let $\kappa_p = \frac{p-4}{p-2}$. Then it holds
\begin{equation}
\label{eq:PartialDecoupling}
A_p(m,Q^r_{\mu},F) \leq C_{\delta} 2^{r \delta} (C \log(r/m))^c 2^{\frac{2m}{p} \big( \frac{1-\alpha}{2} \big)} A_p(2m,Q^r_{\mu},F)^{1-\kappa_p} \tilde{D}_p(m,Q^r_{\mu},F)^{\kappa_p}
\end{equation}
with $c,C$ uniform in $m, r, Q^r_{\mu}$, and $F$, and
\begin{equation*}
\tilde{D}_p(m,Q^r_{\mu},F) = \big( \sum_{J \in \mathbb{I}_m} \| P_J F \|^2_{L^p_{\sharp}(w_{3 Q^r_{\mu}}^J)} \big)^{\frac{1}{2}}.
\end{equation*}
\end{proposition}
This extends Proposition \ref{prop:BilinearSquareFunction} to $p \in [4,6]$, compare \cite[Proposition~10.23]{Demeter2020}.

\smallskip

\emph{Proof of Proposition \ref{prop:PartialDecoupling}}
We start with applying H\"older's inequality to find
\begin{equation*}
\begin{split}
&\quad A_p^p(m,Q^{2m}_{\mu},F) \\
&= \frac{1}{w(Q^r)} \sum_{Q^m \in \mathcal{Q}^{m,r}} w(Q^m) \big[ \sum_{J_1 \in \mathbb{I}_m(I_1)} \| P_{J_1} F \|^2_{L^2_{\sharp}(w_{Q^m_{\mu}})} \sum_{J_2 \in \mathbb{I}_m(I_2)} \| P_{J_2} F \|^2_{L^2_{\sharp}(w_{Q^m_{\mu}})} \big]^{\frac{p}{4}} \\
&\lesssim \frac{1}{w(Q^r)} \sum_{Q^m \in \mathcal{Q}^{m,r}} w(Q^m) \big[ \sum_{J_1 \in \mathbb{I}_m(I_1)} \| P_{J_1} F \|^2_{L^{\frac{p}{2}}_{\sharp}(w_{Q^m_{\mu}})} \sum_{J_2 \in \mathbb{I}_m(I_2)} \| P_{J_2} F \|^2_{L^{\frac{p}{2}}_{\sharp}(w_{Q^m_{\mu}})} \big]^{\frac{p}{4}}.
\end{split}
\end{equation*}
Next, this is expanded as
\begin{equation}
\label{eq:NestedApp}
    \begin{split}
        &\quad A_p^p(m,Q^r_{\mu},F) \\
        &\lesssim w(Q^r)^{-1} \sum_{Q^{2m} \in \mathcal{Q}^{2m,r}} w(Q^{2m}) \frac{1}{w(Q^{2m})} \\
        &\quad \sum_{Q^m \in \mathcal{Q}^{m,2m}} w(Q^m) \big[ \sum_{J_1 \in \mathbb{I}_m(I_1)} \| P_{J_1} F \|^2_{L^{\frac{p}{2}}_{\sharp}(w_{Q^m_{\mu}})} \sum_{J_2 \in \mathbb{I}_m(I_2)} \| P_{J_2} F \|^2_{L^{\frac{p}{2}}_{\sharp}(w_{Q^m_{\mu}})}  ]^{\frac{p}{4}}.
        \end{split}
\end{equation}

For $2 m \leq r/2$ we analyze the local quantity
\begin{equation*}
(A_p^{loc}(Q^{2m}_{\mu},F))^p = \frac{1}{w(Q^{2m})} \sum_{Q^m \in \mathcal{Q}^{m,2m}(Q^{2m})} w(Q^m) \big[ \prod_{i=1}^2 \sum_{J_i \in \mathbb{I}_m(I_i)} \| P_{J_i} F \|^2_{L^{\frac{p}{2}}_{\sharp}(w_{Q^m_\mu})} \big]^{\frac{p}{4}}
\end{equation*}
with the understanding that the weight is $\alpha$-admissible on the larger ambient scale $2^r$. This is the content of Lemma \ref{lem:PartialDecouplingAux}. 
\begin{lemma}
\label{lem:PartialDecouplingAux}
The following holds:
\begin{equation*}
\begin{split}
    &\quad A^{loc}_p(Q^{2m}_{\mu},F) \leq C_\delta 2^{r \delta} 2^{\frac{2m}{p} \big( \frac{1-\alpha}{2} \big) } \times \\
    &\quad \big( \sum_{J_1 \in \mathbb{I}_m(I_1)} \| P_{J_i} F \|^{2}_{L^{p/2}_{\sharp}(w_{3 Q^{2m}_{\mu}}^{J_1})} \sum_{J_2 \in \mathbb{I}_m(I_1)} \| P_{J_2} F \|^{2}_{L^{p/2}_{\sharp}(w_{3 Q^{2m}_{\mu}}^{J_2})} \big)^{1/4}
\end{split}
\end{equation*}
for weights $w_{3 Q^{2m}_{\mu}}^{J_i}$ defined in \eqref{eq:GeneralizedWeight}. 
\end{lemma}

\begin{proof}[Proof~of~Lemma~\ref{lem:PartialDecouplingAux}]
For $I \in \mathbb{I}_m$ the weight $w^I_{3 Q^{2m}_{\mu}}$ satisfies $w_{Q^m_{\mu}} \leq w^I_{3 Q^{2m}_{\mu}}$. For this reason, 
\begin{equation*}
\max_{Q^m \in \mathcal{Q}^{m,r}} \| P_{J_i'} F \|_{L^p_{\sharp}(w_{Q^m_{\mu}})} \leq 2^{r} \max_{J_i \in \mathbb{I}_m(I_i)} \| P_{J_i} F \|_{L^p_{\sharp}(w^{J_i}_{3 Q^{2m}_{\mu}})}.
\end{equation*}
Consequently, for
\begin{equation*}
\| P_{J_i'} F \|_{L^{\frac{p}{2}}_{\sharp}(w_{Q^m_{\mu}})} \leq 2^{- C r} \max_{J_i \in \mathbb{I}_m} \| P_{J_i} F \|_{L^{\frac{p}{2}}_{\sharp}(w^{J_i}_{3 Q^{2m}_{\mu}})}
\end{equation*}
and $C \geq 1$ chosen large enough, we can estimate this contribution to $A_p^p(m,r,Q^r_{\mu})$ directly by the triangle inequality.
 
We can suppose that $\| P_{J_i} F \|_{L^{\frac{p}{2}}_{\sharp}(w_{3 Q^{2m}_{\mu}}^{J_i})}$ ranges over $c r$ dyadic scales because
\begin{equation*}
\| P_{J_i} F \|_{L^{\frac{p}{2}}_{\sharp}(w_{Q^m_{\mu}})} \leq \| P_{J_i} F \|_{L^{\frac{p}{2}}_{\sharp}(w^{J_i}_{3 Q^{2m}_{\mu}})} .
\end{equation*}
(Notably, the constant $c$ can be chosen independent of $Q^{2m}$.)
For this reason, we can choose scales $S_i$ such that 
\begin{equation*}
\sum_{J_i \in \mathbb{I}_m(I_i)} \| P_{J_i} F \|^2_{L^{\frac{p}{2}}_{\sharp}(w_{Q^m_{\mu}})} \leq C_\delta 2^{r \delta} \sum_{J_i \in \mathbb{I}'_m(I_i)} \| P_{J_i} F \|^2_{L^{\frac{p}{2}}_{\sharp}(w_{Q^m_{\mu}})}
\end{equation*}
for $\mathbb{I}_m'(I_i) \subseteq \mathbb{I}_m(I_i)$ with $\| P_{J_i} F \|_{L^{\frac{p}{2}}_{\sharp}(w^{J_i}_{Q^{2m}_{\mu}})} \sim S_i$.
Note that the scales $S_i$ may vary over $Q^{2m} \in \mathcal{Q}^{2m,r}$. We find for $A_p^{loc}$:
\begin{equation*}
\begin{split}
(A_p^{loc})^p &\leq C_{\delta} 2^{r \delta} \frac{1}{w(Q^{2m})} \sum_{Q^m \in \mathcal{Q}^{m,2m}(Q^{2m})} w(Q^m) \big[ \prod_{i=1}^2 \sum_{J_i \in \mathbb{I}_m'(I_i)} \| P_{J_i} F \|^2_{L^{\frac{p}{2}}_{\sharp}(w_{Q^{m}_{\mu}})} \big]^{\frac{p}{4}} \\
&\quad + 2^{-cr} \big( \prod_{i=1}^2 \sum_{J_i \in \mathbb{I}_m(I_i)} \| P_{J_i} F \|^2_{L^{p/2}_{\sharp}(w_{3Q^{2m}_{\mu}})} \big). 
\end{split}
\end{equation*}
In the following we estimate the first term. The remainder term is already acceptable.
Let $N_i$ be the number of intervals $J_i \in \mathbb{I}'_m(I_i)$, over which we are summing in the above display.
We can apply H\"older's inequality to find
\begin{equation*}
\begin{split}
&\quad \frac{1}{w(Q^{2m})} \sum_{Q^m \in \mathcal{Q}^{m,2m}} \big[ \prod_{i=1}^2 \big( \sum_{J_i \in \mathbb{I}'_m(I_i)} \| P_{J_i} F \|^2_{L^{\frac{p}{2}}_{\sharp}(w_{Q^m_{\mu}})} \big)^{\frac{1}{2}} \big]^{\frac{p}{2}} \\
&\lesssim \frac{1}{w(Q^{2m})} \prod_{j=1}^2 N_j^{\frac{p}{2} \big( \frac{1}{2} - \frac{2}{p} \big)} \sum_{Q^m \in \mathcal{Q}^{m,2m}} w(Q^m) \big[ \prod_{i=1}^2 \big( \sum_{J_i \in \mathbb{I}'_m(I_i)} \| P_{J_i} F \|^{\frac{p}{2}}_{L^{\frac{p}{2}}_{\sharp}(w_{Q^m_{\mu}})} \big) \big] \\
&= \frac{1}{w(Q^{2m})} \prod_{j=1}^2 N_j^{\frac{p}{2} \big( \frac{1}{2} - \frac{2}{p} \big)} \sum_{Q^m \in \mathcal{Q}^{m,2m}} w(Q^m) \big[ \prod_{i=1}^2 \dashint \sum_{J_i} |P_{J_i} F |^{\frac{p}{2}} w_{Q^m_{\mu}} \big].
\end{split}
\end{equation*}
We use the scattered modulation sum to rewrite with $\sigma = 2^{-m}$ for $x \in Q^m$:
\begin{equation*}
\begin{split}
\frac{1}{w(Q^m)} \int_{\R^2} \sum_{J_i \in \mathbb{I}'_m} |P_{J_i} F|^{\frac{p}{2}} w_{Q^m_{\mu}} &\lesssim \sum_{k \in \Z^2} (1+|k|)^{-100} \sum_{J_i \in \mathbb{I}'_m} [P_{J_i} F]^{\frac{p}{2}}_{\sigma}(x_q + k 2^m) \\
&\lesssim \sum_{J_i \in \mathbb{I}'_m} \sum_{k \in \Z^2} (1+|k|)^{-100} [[ P_{J_i} F ]]^{\frac{p}{2}}_{\sigma}(x+k 2^m) \\
&\lesssim \sum_{J_i \in \mathbb{I}'_m} \sum_{k \in \Z^2} (1+|k|)^{-100} | P_{J_i} F |^{\frac{p}{2}}(x+ k 2^m).
\end{split}
\end{equation*}
For brevity we let $\sum_{k \in \Z^2} B_k | P_{J_i} F |^{\frac{p}{2}}(x+ k 2^m) =: [ P_{J_i} F ]_{\sigma,100}^{p/2}(x)$, $B_k = (1+|k|)^{-100}$. Next, we average over $x \in Q^m$ to find
\begin{equation*}
\dashint_{Q^m_{\mu}} \big[ \prod_{i=1}^2 \frac{1}{w(Q^m)} \int \sum_{J_i} |P_{J_i} F|^{\frac{p}{2}} w_{Q^m_{\mu}} \big] \lesssim \dashint_{Q^m_{\mu}} \prod_{i=1}^2 \sum_{J_i \in \mathbb{I}'_m(I_i)} [P_{J_i} F]^{p/2}_{\sigma,100}(x)  dx.
\end{equation*}
Carrying out the sum over $Q^m \in \mathcal{Q}^{m,2m}$, we find
\begin{equation*}
\begin{split}
    &\quad \frac{1}{w(Q^{2m})} \sum_{Q^m \in \mathcal{Q}^{m,2m}} w(Q^m) \big[ \prod_{i=1}^2 \frac{1}{w(Q^m)} \int \sum_{J_i} |P_{J_i} F |^{\frac{p}{2}} w_{Q^m_{\mu}} \big] \\
    &\lesssim \dashint_{Q^{2m}_{\mu}} \sum_{J_1 \in \mathbb{I}'_m(I_1)} \big[ P_{J_1} F \big]^{p/2}_{\sigma,100}(x) \sum_{J_2 \in \mathbb{I}'_m(I_2)} \big[ P_{J_2} F \big]^{p/2}_{\sigma,100}(x) dx.
\end{split}
\end{equation*}
To bring the fractal Kakeya estimate into play, for given $J_i$ we tile $Q^{2m}$ with tubes $T \in \mathbb{T}_{J_i}$ of size $2^m \times 2^{2m}$ aligned dually to $\theta \supseteq \text{supp}(\widehat{P_{J_i} F})$, $\theta$ a $2^{-m} \times 2^{-2m}$-rectangle into directions $(e_\theta,n_{\theta})$. Then, we define
\begin{equation*}
    G_{J_i}^{p/2}(x) = \max_{y \in 2 T_{J_i}(x)} [ P_{J_i} F ]^{p/2}_{\sigma,100}(y),
\end{equation*}
with $T_{J_i}(x) \in \mathbb{T}_{J_i}$ denoting the tube encasing $x$. These functions are constant on tubes $T_{J_i} \in \mathbb{T}_{J_i}$ and we find applying Proposition \ref{prop:FractalKakeya} by hypothesis \eqref{eq:LocalAhlfors}
\begin{equation*}
    \begin{split}
        &\quad \dashint_{Q^{2m}_{\mu}} \sum_{J_1 \in \mathbb{I}'_m(I_1)} [P_{J_1} F]^{p/2}_{\sigma,100}(x) \sum_{J_2 \in \mathbb{I}'_m(I_2)} [P_{J_2} F]^{p/2}_{\sigma,100}(x) dx \\
        &\lesssim \dashint_{Q^{2m}_{\mu}} \sum_{J_1} G_{J_1}^{p/2}(x) \sum_{J_2} G_{J_2}^{p/2}(x) dx \\
        &\lesssim 2^{2m( \frac{1-\alpha}{2} )} \dashint_{3 Q^{2m}_{\mu}} \sum_{J_1} G_{J_1}^{p/2}(x) dx \dashint_{3 Q^{2m}_{\mu}} \sum_{J_2} G_{J_2}^{p/2}(x) dx.
    \end{split}
\end{equation*}

To compare $G_{J_i}$ with $P_{J_i} F$ we use an anisotropic scattered modulation sum. Indeed, by anisotropic rescaling, Lemma \ref{lem:ScatteredModulationSum} yields with $A_l = (1+|l|)^{-200}$
\begin{equation*}
  \max_{y \in 2 T_{J_i}(x)}  |P_{J_i} F|(y) \leq C \sum_{l \in \Z^2} A_l |P_{J_i} F|(x+2^m e_{\theta} l_1 + 2^{2m} n_{\theta} l_2). 
\end{equation*}
Applying Jensen's inequality gives
\begin{equation*}
    [P_{J_i} F]^{p/2}_{\sigma,100}(y) \lesssim \sum_{k \in \Z^2} B_k |P_{J_i} F|^{p/2}(y+k 2^m).
\end{equation*}
Another application of Jensen's inequality combined with the scattered anisotropic modulation sum estimate, we find
\begin{equation*}
    \max_{y \in 2 T_{J_i}(x)} [P_{J_i} F]^{p/2}_{\sigma,100}(y) \lesssim \sum_{k \in \Z^2} B_k \sum_{l \in \Z^2} A_l |P_{J_i} F|^{p/2}(x+k 2^m + l_1 2^m e_{\theta} + l_2 2^{2m} n_{\theta}).
\end{equation*}
By translation, we find
\begin{equation*}
    \begin{split}
    &\quad \int_{\R^2} \sum_{k \in \Z^2} \sum_{l \in \Z^2} B_k A_l |P_{J_i} F|^{p/2}(x+k 2^m + l_1 2^m e_{\theta} + l_2 2^{2m} n_{\theta}) (w \chi_{3 Q^{2m}})(x)  dx \\
    &= \sum_{k,l \in \Z^2} B_k A_l \int_{\R^2} |P_{J_i} F|^{p/2}(x) (w \chi_{3 Q^{2m}})(x +k 2^m + l_1 2^m e_{\theta} + l_2 2^{2m} n_{\theta}) dx \\
    &=: \int_{\R^2} |P_{J_i} F|^{p/2}(x) w_{3 Q^{2m}_{\mu}}^{J_i} dx
    \end{split}
\end{equation*}
Since $Q^m \in \mathcal{Q}^{m,r}$, that is $Q^m \subseteq Q^{2m} \subseteq 3 Q^{2m}$, this weight indeed satisfies
$ w_{Q^m_{\mu}} \leq w_{3Q^{2m}_{\mu}}^{J_i}$. In conclusion,
\begin{equation}
\label{eq:PartialDecouplingAuxIII}
\begin{split}
    &\quad w(Q^{2m})^{-1} \sum_{Q^m \in \mathcal{Q}^{m,r}} w(Q^m) \big[ \sum_{J_1 \in \mathbb{I}'_m(I_1)} \| P_{J_1} F \|^2_{L^{p/2}_{\sharp}(w_{Q^m_{\mu}})} \sum_{J_2 \in \mathbb{I}'_m(I_2)} \| P_{J_2} F \|^2_{L^{p/2}_{\sharp}(w_{Q^m_{\mu}})}  ]^{\frac{p}{2}} \\
    &\leq C 2^{2m ( \frac{1-\alpha}{2} )} \prod_{i=1}^2 N_i^{\frac{p}{2} (\frac{1}{2} - \frac{2}{p}) } \sum_{J_1 \in \mathbb{I}'_m(I_1)} \|P_{J_1} F \|^{p/2}_{L^{p/2}_{\sharp}(w_{3 Q^{2m}_{\mu}}^{J_1})} \sum_{J_2 \in \mathbb{I}'_m(I_2)} \|P_{J_2} F \|^{p/2}_{L^{p/2}_{\sharp}(w_{3Q^{2m}_{\mu}}^{J_2})}.
    \end{split}
\end{equation}
By the comparability $\| P_{J_i} F \|_{L^{p/2}_{\sharp}(w_{3 Q^{2m}_{\mu}}^{J_i})} \sim S_i$ and the number of $J_i \in \mathbb{I}'_m(I_i)$ in the restricted sum being comparable to $N_i$, we can reverse H\"older's inequality to find
\begin{equation*}
  \eqref{eq:PartialDecouplingAuxIII}  \leq C 2^{2m \big( \frac{1-\alpha}{2} \big)} \big( \sum_{J_1 \in \mathbb{I}_m(I_1)} \| P_{J_i} F \|^{2}_{L^{p/2}_{\sharp}(w_{3 Q^{2m}_{\mu}}^{J_i})} \sum_{J_2 \in \mathbb{I}_m(I_1)} \| P_{J_2} F \|^{2}_{L^{p/2}_{\sharp}(w_{3 Q^{2m}_{\mu}}^{J_2})} \big)^{p/4}. 
\end{equation*}
Taking the $p$th root finishes the proof of Lemma \ref{lem:PartialDecouplingAux}.
\end{proof}

The following identity follows from two applications of H\"older's inequality:
\begin{equation}
    \label{eq:PartialDecouplingAuxII}
    \begin{split}
    &\quad \big( \sum_{J_1 \in \mathbb{I}_m(I_1)} \| P_{J_1} F \|^2_{L^{\frac{p}{2}}_{\sharp}(w_{3 Q^{2m}_{\mu}}^{J_1})} \sum_{J_2 \in \mathbb{I}_m(I_2)} \| P_{J_2} F \|^2_{L^{\frac{p}{2}}_{\sharp}(w_{3 Q^{2m}_{\mu}}^{J_1})} \big)^{\frac{1}{4}} \\
    &\lesssim \big( \sum_{J_1 \in \mathbb{I}_m(I_1)} \| P_{J_1} F \|^2_{L^2_{\sharp}(w_{3 Q^{2m}_{\mu}}^{J_1})} \sum_{J_2 \in \mathbb{I}_m(I_2)} \| P_{J_2} F \|^2_{L^2_{\sharp}(w_{3 Q^{2m}_{\mu}}^{J_2})} \big)^{\frac{1-\kappa_p}{4}} \\ 
    &\quad \times \big( \sum_{J_1 \in \mathbb{I}_m(I_1)} \| P_{J_1} F \|^2_{L^p_{\sharp}(w_{3 Q^{2m}_{\mu}}^{J_1})} \sum_{J_2 \in \mathbb{I}_m(I_2)} \| P_{J_2} F \|^2_{L^p_{\sharp}(w_{3 Q^{2m}_{\mu}}^{J_2})} \big)^{\frac{\kappa_p}{4}}. 
    \end{split}
\end{equation}

To further process the $L^2$-averages, we argue similarly to the previous section. We have the following:
\begin{lemma}[Almost~orthogonality~for~fractal~weights]
\label{lem:AlmostOrthogonalityFractal}
   Let $r=2m$, $J_i \in \mathbb{I}_m$ and with remaining notation like above, the following holds:
    \begin{equation}
        \label{eq:LocalOrthogonality}
        \| P_{J_i} F \|^2_{L^2_{\sharp}(w_{3 Q^r_{\mu}}^{J_i})} \lesssim \sum_{J_i' \in \mathbb{I}_{2m}(J_i)} \| P_{J_i'} F \|^2_{L^2_{\sharp}(w_{Q^r_{\mu}})}.
    \end{equation}
\end{lemma}
\begin{proof}
In the first step we recast the fractal weight as the usual weight.
We break the lefthand side into dyadic cubes of size $2^r$. For the centered cube $Q^r$ and cubes $Q^r + n 2^r$, $|n| \leq 10$, we shall prove
\begin{equation}
\label{eq:FractalAlmostOrthogonalityAux}
\begin{split}
&\; \frac{1}{w(Q^r)} \int_{Q^r+n 2^r} \sum_{k,l \in \Z^2} (1+|k|)^{-100} (1+|l|)^{-200} \\
&\quad  (w \chi_{3 Q^r})(x-k2^m - l_1 e_{\theta} 2^m - l_2 n_{\theta}2^{2m}) |P_{J_i} F|^2(x) dx \lesssim \| P_{J_i} F \|^2_{L^2_{\sharp}(w_{Q^r}^M)},
\end{split}
\end{equation}
which suffices by $w(3 Q^r) \sim w(Q^r)$ by virtue of \eqref{eq:MildContinuity} and \eqref{eq:AlmostProductWeights}. For the sake of brevity let $\Delta_{k,l} = k2^m + l_1 e_{\theta} 2^m + l_2 n_{\theta}2^{2m}$.

To this end, we plug in the wave packet decomposition.
Let $\T$ be a tiling of $\R^2$ with rectangles of size $2^m \times 2^{2m}$ aligned dual to $\text{supp}(\widehat{P_{J_i} F})$. 
With notation from Section \ref{section:Preliminaries}, we find after applying Jensen's inequality
    \begin{equation*}
        |P_{J_1} F|^2(x) \lesssim \sum_T |\alpha_T|^2 2^{-\frac{3m}{2}} |\phi_T(x)|
    \end{equation*}
    and $\phi_T$ being $L^1$-normalized and $|\phi_T(x)| \lesssim_N 2^{-\frac{3m}{2}} (1+ d_T(x))^{-N}$ by \eqref{eq:PointwiseBound}.

    We further break
    \begin{equation*}
        |P_{J_1} F|^2(x) \lesssim \sum_{T : \text{dist}(T,Q^r) \lesssim 2^r} |\alpha_T|^2 2^{-\frac{3m}{2}} |\phi_T(x)| + \sum_{A \geq 16} \sum_{T: \text{dist}(T,Q^r) \sim A 2^r} |\alpha_T|^2 2^{-\frac{3m}{2}} |\phi_T|.
    \end{equation*}
First estimate the tubes intersecting $16 Q^r$. By the estimate
\begin{equation*}
|\phi_T(x)| \lesssim 2^{-\frac{3m}{2}} \sum_{\tilde{T} \parallel T} 1_{\tilde{T}}(x) (1+\text{dist}_{\text{par}}(\tilde{T},T))^{-M},
\end{equation*}
we find by appealing to Lemma \ref{lem:WavePacketLocalization} for $|n| \leq 10$
\begin{equation*}
\begin{split}
&\quad \frac{1}{w(Q^r)} \int_{(Q^r + n 2^r) \cap (3 Q^r+\Delta_{k,l})} \sum_{T: T \cap 16 Q^r \neq \emptyset } |\alpha_T|^2 2^{-\frac{3m}{2}} 2^{-\frac{3m}{2}} \\
&\quad \quad \times \sum_{\tilde{T} \parallel T} 1_{\tilde{T}}(x) (1+\text{dist}_{\text{par}}(\tilde{T},T))^{-M} w(x- \Delta_{k,l}) dx \lesssim 2^{-4m} \| P_{J_i} F\|^2_{L^2(w_{Q^r}^M)}
\end{split}
\end{equation*}
for a large integer $M \geq 200$. The tubes with $A 2^r \leq  \text{dist}(T,Q^r) \leq 2 A 2^r$, $A \geq 16$ can be estimated by similar means, see the proof of Proposition \ref{prop:BilinearSquareFunction}. Here we use that $\text{dist}(Q^r + n 2^r,T) \sim A 2^r$ for $|n| \leq 10$ and \eqref{eq:FractalAlmostOrthogonalityAux} follows. 

Next, we consider for $|n| \geq 10$
\begin{equation*}
\frac{1}{w(Q^r)} \int_{(Q^r + n 2^r) \cap (3 Q^r+\Delta_{k,l})} w(x-\Delta_{k,l}) |P_{J_i} F|^2(x) dx.
\end{equation*}
Let $A^n_{k,l} = (Q^r + n 2^r) \cap (3 Q^r+\Delta_{k,l})$ and for $A_{k,l}^n \neq \emptyset$ we find by repeating the above steps
\begin{equation*}
\frac{1}{w(Q^r)} \int_{(Q^r + n 2^r) \cap (3 Q^r+\Delta_{k,l})} w(x-\Delta_{k,l}) |P_{J_i} F|^2(x) dx \lesssim \| P_{J_i} F \|^2_{L^2_{\sharp}(w^M_{Q^r+n2^r})}.
\end{equation*}
We evaluate the sum over $k,l$ such that $A^n_{k,l} \neq \emptyset$. To this end, observe that $\{ |k| \gtrsim |n| 2^m \vee |l| \gtrsim |n| \}$ must hold true for $A^n_{k,l} \neq \emptyset$. For fixed $l$, there are only $\lesssim 2^{2m}$ possible $k$ for which $A^n_{k,l} \neq \emptyset$. We find 
\begin{equation*}
\begin{split}
\sum_{|k| \gtrsim |n| 2^m, \, l: A^n_{k,l} \neq \emptyset} (1+|k|)^{-100} (1+|l|)^{-200} &\lesssim \sum_{l \geq 0} (1+|l|)^{-200} (1+|n| 2^m)^{-100} 2^{2m} \\
 &\lesssim (1+|n| )^{-100} 2^{-98m}.
 \end{split}
\end{equation*}
It is straightforward that $\sum_{|l| \gtrsim |n|, A^n_{k,l} \neq \emptyset} (1+|k|)^{-100} (1+|l|)^{-200} \lesssim (1+|n|)^{-198}$. 
Consequently,
\begin{equation*}
\begin{split}
&\quad \sum_{k,l \in \Z^2} (1+|k|)^{-100} (1+|l|)^{-200} \frac{1}{w(Q^r)} \int_{A^n_{k,l}} |P_{J_i} F|^2(x) dx \\
&\lesssim (1+|n|)^{-100} \| P_{J_i} F \|^2_{L^2_{\sharp}(w^M_{Q^r + n 2^r})}.
\end{split}
\end{equation*}

We sum the weights as
\begin{equation*}
\sum_{n \in \Z^2} (1+|n|)^{-100} \| P_{J_i} F \|^2_{L^2_{\sharp}(w^M_{Q^r + n 2^r})} \lesssim \| P_{J_i} F \|^2_{L^2_{\sharp}(w_{Q^r})}.
\end{equation*}
%
Now we use local $L^2$-orthogonality provided by Lemma \ref{lem:LocalL2Orthogonality}:
\begin{equation*}
\| P_{J_i} F \|^2_{L^2_{\sharp}(w_{Q^r})} \lesssim \sum_{\substack{I \in \mathbb{I}_{2m}, \\ I \subseteq J_i}} \| P_I F \|^2_{L^2_{\sharp}(w_{Q^r})}.
\end{equation*}
%
%
%
    Finally, we use the locally constant property to recast the $L^2_{\sharp}(w_{Q^r})$-norm as $L^2_{\sharp}(w_{Q_{\mu}^r})$, which was already carried out at the end of the proof of Proposition \ref{prop:BilinearSquareFunction}.
\end{proof}

\begin{proof}[Proof~of~Proposition~\ref{prop:PartialDecoupling},~Conclusion]
We plug the estimate from Lemma \ref{lem:PartialDecouplingAux} into \eqref{eq:NestedApp}, use \eqref{eq:PartialDecouplingAuxII}, and apply Lemma \ref{lem:AlmostOrthogonalityFractal} to continue
\begin{equation*}
\begin{split}
       \eqref{eq:NestedApp} &\lesssim_{\delta} 2^{2m ( \frac{1-\alpha}{2} )+r \delta} \frac{1}{w(Q^r)} \sum_{Q^{2m} \in \mathcal{Q}^{2m,r}} w(Q^{2m}) \\
        &\quad \times \big[ \sum_{J_1 \in \mathbb{I}_{2m}(I_1)} \| P_{J_1} F \|^2_{L^2_{\sharp}(w_{Q^{2m}_{\mu}})} \sum_{J_2 \in \mathbb{I}_{2m}(I_2)} \| P_{J_1} F \|^2_{L^2_{\sharp}(w_{Q^{2m}_{\mu}})} \big]^{(1-\kappa_p) \frac{p}{4}} \\ 
        &\quad \times \big[ \sum_{J_1 \in \mathbb{I}_{m}(I_1)} \| P_{J_1} F \|^2_{L^p_{\sharp}(w^{J_1}_{Q^{2m}_{\mu}})} \sum_{J_2 \in \mathbb{I}_{m}(I_2)} \| P_{J_2} F \|^2_{L^p_{\sharp}(w^{J_2}_{Q^{2m}_{\mu}})} \big]^{\kappa_p \cdot \frac{p}{4}} \\
     &\lesssim_{\delta} 2^{2m ( \frac{1-\alpha}{2} )+r \delta} A_p^{(1-\kappa_p) p}(2m,Q^r_{\mu},F) \\
        &\, \cdot \bigg[ \frac{1}{w(Q^r)} \sum_{Q^{2m} \in \mathcal{Q}^{2m,r}} w(Q^{2m}) \big[ \prod_{i=1}^2 \sum_{J_i \in \mathbb{I}_m(I_i)} \| P_{J_i} F \|^2_{L^p_{\sharp}(w^{J_i}_{3Q^{2m}_{\mu}})}  \big]^{\frac{p}{4}}  \bigg]^{\kappa_p} \\
        &\lesssim_{\delta} 2^{2m ( \frac{1-\alpha}{2} )+r \delta}  (C\log(r/m))^c A_p^{(1-\kappa_p)p}(2m,Q^r_{\mu},F) \tilde{D}_p^{p \cdot \kappa_p}(2m,Q^r_{\mu},F).
    \end{split}
\end{equation*}
The ultimate estimate follows from Minkowski's inequality and the  overlap of $Q^m \in \mathcal{Q}^{m,r}$. The finite overlap at any scale $Q^l \in \mathcal{Q}^{l,2l}$ accumulates an overlap $C \log(r/m)$. Taking the $p$th root finishes the proof of Proposition \ref{prop:PartialDecoupling}.
\end{proof}

\subsection{Iterating the partial decoupling}

We start with iterating Proposition \ref{prop:PartialDecoupling} to obtain a local multi-scale inequality for $A_6((r/2)/2^s,Q^r_{\mu},F)$. Recall that we stop at interval length $2^{-r/2}$. So we will iterate Proposition \ref{prop:PartialDecoupling} $s$ times.

\begin{proposition}[Multi-scale inequality]
\label{prop:MultiScale}
The following estimate holds:
\begin{equation}
\label{eq:LocalMultiscale}
\begin{split}
&\quad A_6((r/2)/2^s,Q^r_{\mu},F) \\
&\leq C_{\delta,s} 2^{\frac{r}{6 \cdot 2^s}(\frac{1-\alpha}{2} ) + r \delta} A_6^{\frac{1}{2}}((r/2)/2^{s-1},Q^r_{\mu},F) \tilde{D}_6^{\frac{1}{2}}((r/2)/2^s,Q^r_{\mu},F) \\
&\leq \ldots \\
&\leq C_{\delta,s}^2 2^{\frac{r}{6 \cdot 2^s}(\frac{1-\alpha}{2})s + 2r \delta} A_6^{\frac{1}{2^{s}}}(r/2,r,Q^r_{\mu}) \tilde{D}_6^{\frac{1}{2}}((r/2)/2^s,Q^r_{\mu},F) \ldots \tilde{D}_6(r/4,Q^r_{\mu},F)^{\frac{1}{2^{s}}}.
\end{split}
\end{equation}
\end{proposition}
The dependence of $C_{\delta,s}$ on $s$ stems from the logarithmic factor in Proposition \ref{prop:PartialDecoupling} governed by the bi-dyadic difference of $r$ and $m$. In the present context this is $\sim s^c$.
Let $C(r,s,\delta) = C_{\delta,s}^2 2^{\frac{r}{6 \cdot 2^s} \big( \frac{1-\alpha}{2} \big) s + 2r \delta }$. 
We further process \eqref{eq:LocalMultiscale} applying H\"older's inequality
\begin{equation*}
A_6(r/2,Q^r_{\mu},F) \lesssim \big( \sum_{I \in \mathbb{I}_{r/2}} \| P_I F \|^2_{L^p_{\sharp}(w^I_{Q^r_{\mu}})} \big)^{\frac{1}{2}}
\end{equation*}
with $w^I_{Q^r_{\mu}} = \sum_{Q^m_{\mu} \in \mathcal{Q}^{m,r}_{\mu}(Q^r_{\mu})} w_{Q^m_{\mu}}$ for $I \in \mathbb{I}_{r/2}$. For $I \in \mathbb{I}_a$, $a = r/2^s$, $s \geq 1$, we define $w^I = \sum_{Q^r \in \mathcal{Q}^r} w^I_{3 Q^r_{\mu}}$.
We obtain a global multi-scale inequality from \eqref{eq:LocalMultiscale}, by raising it to the $p$th power, multiplying with $w(Q^r)$, summing over $Q^r \in \mathcal{Q}^r(\R^2)$, and applying H\"older's inequality
\begin{equation*}
\begin{split}
&\quad \| |F_1 F_2|^{\frac{1}{2}} \|_{L^6(w)} \\
&\leq C(r,s,\delta) \big( \sum_{I \in \mathbb{I}_{r/2}} \| P_I F \|^2_{L^6(w^I)} \big)^{\frac{1}{2} \cdot \frac{1}{2^{s}}} \ldots \big( \sum_{I \in \mathbb{I}_{(r/2)/2^s}} \| P_I F \|^2_{L^6(w^I)} \big)^{\frac{1}{2} \cdot \frac{1}{2}}.
\end{split}
\end{equation*}

In the next step we use parabolic rescaling to find for $a = r/2^s$, $s \geq 2$
\begin{equation*}
\big( \sum_{I \in \mathbb{I}_a} \| P_I F \|^2_{L^6(w^I)} \big)^{\frac{1}{2}} \lesssim \text{Dec}(6,r-2a,\alpha) \big( \sum_{I \in \mathbb{I}_{r/2}} \| P_I F \|^2_{L^6(w^{I,F})} \big)^{\frac{1}{2}}
\end{equation*}
with admissible weights $w^{I,F}$. Note that the induction hypothesis is not directly applicable to $L^6(w^I)$. However, after expanding $w^I$ into weighted translates, we can absorb the rapidly decaying coefficients into the admissible weights obtained from applying the induction hypothesis.

\subsection{Bootstrapping the global decoupling constant through the multi-scale inequality}
In the following we bootstrap the decoupling constant:
\begin{equation*}
\text{Dec}(6,r,\alpha) = C_\varepsilon 2^{r \big( \frac{1-\alpha}{6} \big) + r \varepsilon}.
\end{equation*}
By Cauchy-Schwarz we can allow for $\varepsilon = \frac{1}{2}$, $C_\varepsilon = 1$. The bootstrap argument will yield that we can choose $\varepsilon$ arbitrarily small. So, in the following we take the a priori bound $C_\varepsilon 2^{r \big( \frac{1-\alpha}{6} \big) + r \varepsilon}$ for granted and our goal is to decrease $\varepsilon$. 

By Lemma \ref{lem:LinearBilinearReduction}, it suffices to bound $\text{BilDec}(6,r,\alpha)$, which will be done through the multi-scale inequality in Proposition \ref{prop:MultiScale}. 
To this end, we choose $s$ minimally such that $s \geq \frac{4}{\varepsilon}$. If $2^s \geq R/16$, then $R=R(\varepsilon)$ and it suffices to choose $C_\varepsilon$ large enough.
Plugging in the a priori bound for $Dec$, we find
\begin{equation*}
\begin{split}
&\quad \big( \sum_{I \in \mathbb{I}_{(r/2)/2^s}} \| P_I F \|^2_{L^6(w^I)} \big)^{\frac{1}{2} \cdot \frac{1}{2}} \ldots \big( \sum_{I \in \mathbb{I}_{(r/2)/2}} \| P_I F \|^2_{L^6(w^I)} \big)^{\frac{1}{2} \cdot \frac{1}{2^s}} \\
&\leq 2^{(r-\frac{r}{2^s})(\frac{1-\alpha}{6}) \frac{1}{2}} 2^{(r-\frac{r}{2^{s-1}})(\frac{1-\alpha}{6})\frac{1}{4}} \ldots \\
&\quad 2^{(r-\frac{r}{2}2^{-s}) \frac{\varepsilon}{2}} 2^{(r-\frac{r}{2}2^{-(s-1)})\frac{\varepsilon}{4}} \ldots 2^{(r-r/2)\frac{\varepsilon}{2^s}} C_\varepsilon^{1-2^{-s}} \big( \sum_{I \in \mathbb{I}_{r/2}} \| P_I F \|^2_{L^6(w^{I,F})} \big)^{\frac{1}{2} (1-2^{-s})} \\
&\leq 2^{(\frac{1-\alpha}{6})r (1-2^{-s})} 2^{-\frac{r}{2} (\frac{1-\alpha}{6}) \frac{s}{2^s}} 2^{r(1-2^{-s})\varepsilon} 2^{-\frac{r}{2} \varepsilon \frac{s}{2^s}} C_\varepsilon^{1-2^{-s}} \big( \sum_{I \in \mathbb{I}_{r/2}} \| P_I F \|^2_{L^6(w^{I,F})} \big)^{\frac{1}{2} (1-2^{-s})}.
\end{split}
\end{equation*}
Consequently, we find for the global bilinear expression, taking into consideration the factor $2^{r/2^s}$ to jump start the decoupling iteration and $C_{\delta}^2 2^{\frac{r/2}{2^s} s \delta}$ from iterating as recorded in Proposition \ref{prop:MultiScale}:
\begin{equation*}
\begin{split}
\| |F_1 F_2|^{\frac{1}{2}} \|_{L^6(w)} &\leq 2^{\frac{r}{2^s}} C_{\delta,s}^2 2^{2 r \delta } 2^{\big( \frac{1-\alpha}{6} \big) r (1-2^{-s})} 2^{r (1-2^{-s}) \varepsilon} 2^{-\frac{r}{2} \varepsilon \frac{s}{2^s}} C_\varepsilon^{1-2^{-s}} \\
&\quad \quad \big( \sum_{I \in \mathbb{I}_{r/2}} \| P_I F \|^2_{L^6(w^{I,F})} \big)^{\frac{1}{2}}.
\end{split}
\end{equation*}
By the linear-to-bilinear reduction, we find for the global linear expression
\begin{equation*}
\begin{split}
\| F \|_{L^6(w)} &\leq C_{\delta'} 2^{r \delta'} C_\delta^2 2^{2 r \delta } (2^{\frac{r}{2^s}} 2^{-\frac{r}{2} \frac{s}{2^s} \varepsilon}) 2^{\big( \frac{1-\alpha}{6} \big) r (1-2^{-s})} 2^{r (1-2^{-s})\varepsilon} C_\varepsilon^{1- 2^{-s}} \\
&\quad \quad \big( \sum_{I \in \mathbb{I}_{r/2}} \| P_I F \|^2_{L^6(w^{I,F})} \big)^{\frac{1}{2}}.
\end{split}
\end{equation*}
By our choice of $s$ we find
\begin{equation*}
(2^{\frac{r}{2^s}} 2^{-\frac{r}{2} \frac{s}{2^s} \varepsilon}) \leq 2^{-r \frac{1}{2^s}}.
\end{equation*}
Choose $\delta = \frac{1}{2^{s+2}} = \delta(\varepsilon)$ and $\delta'= \frac{1}{2^{s+1}} = \delta'(\varepsilon)$ such that
\begin{equation*}
2^{r \delta'} 2^{2r \delta} (2^{\frac{r}{2^s}} 2^{-\frac{r}{2} \frac{s}{2^s} \varepsilon}) \leq 1. 
\end{equation*}
Hence, with $\varepsilon' := (1-2^{-s}) \varepsilon$ and $C_{\varepsilon'} := C_\varepsilon^{1-2^{-s}} C_{\delta'} C_{\delta,s}^2$ (recall $s=s(\varepsilon)$) we find
\begin{equation*}
\| F \|_{L^6(w)} \leq C_{\varepsilon'} 2^{\big( \frac{1-\alpha}{6} \big) r } 2^{r \varepsilon'}  \big( \sum_{I \in \mathbb{I}_{r/2}} \| P_I F \|^2_{L^6(w^{I,F})} \big)^{\frac{1}{2}}.
\end{equation*}
This finishes the proof of Theorem \ref{thm:FractalDecouplingTheorem}.
$\hfill \qed$

\section{Exponential sum estimates via weighted decoupling inequalities}
\label{section:ExpSum}
For the sake of completeness, we record a proof of the exponential sum estimate. We show the fractal $L^4$-estimate: 
\begin{proof}[Proof~of~Theorem~\ref{thm:FractalL4Strichartz}]
We let $(x_1,x_2) = (x,t)$. Let $\chi_{N^2}: \R \to \R$ be a Schwartz function  which satisfies $|\chi_{N^2}(z)| \gtrsim 1$ for $|z| \leq N^2$ and has compact Fourier support in $B_2(0,cN^{-2})$. By non-stationary phase it satisfies the decay estimate $|\chi_{N^2}(z)| \lesssim_m (1+N^{-2} |z|)^{-m}$ for any $m \geq 0$. Let 
\begin{equation*}
e^{it \Delta} P_N f = \frac{1}{2 \pi} \sum_{\substack{|k| \sim N,\\ k \in \Z}} e^{i(kx - t|k|^2)} a_k.
\end{equation*}
To estimate the exponential sum, we rescale $t \to N^2 t$, $x \to N x$ and use the periodicity in $x$ to find
\begin{equation}
\label{eq:ExponentialSumScaling}
\begin{split}
\| e^{it \Delta} P_N f \|^4_{L^4_{t,x}(\mathcal{I} \times \T)} &\lesssim N^{-4} \big\| \sum_{\substack{|k| \sim 1, \\ k \in \Z/N}} e^{i(kx+t|k|^2)} a_{N k} \big\|^4_{L^4_{t,x}(N^2 \mathcal{I} \times N^2 \T)} \\ &\lesssim N^{-4} \big\| \sum_{\substack{|k| \sim 1, \\ k \in \Z/N}} e^{i(kx + tk^2)} a_{N k} \chi_{N^2}(x,t)  \chi_{\bar{I}^{N^2}_{\mu}}(t) \big\|_{L^4_{t,x}(\R^2)}^4.
\end{split}
\end{equation}
The function
\begin{equation*}
F(x,t)= \sum_{\substack{|k| \sim 1,\\ k \in \Z/N}} e^{i(kx + tk^2)} a_{Nk} \chi_{N^2}(x) \chi_{N^2}(t)
\end{equation*}
 satisfies $\text{supp}(\hat{F}) \subseteq \mathcal{N}_{N^{-2}}(\mathbb{P}^1)$ and by hypothesis $w(x,t)= \chi_{\bar{I}^{N^2}_{\mu}}(t)$ is a weight which satisfies the assumptions of Theorem \ref{thm:FractalSquareFunction}. Consequently, we apply Theorem \ref{thm:FractalSquareFunction} and by \eqref{eq:SimplificationLocalAhlfors} we find for $F$ localized to $[0,N^2] \times [0,N^2] = Q^{N^2}$:
\begin{equation}
\label{eq:LocalSF}
\| F \|^4_{L^4_{t,x}(w)} \lesssim \log(N)^c N^{1-\alpha} \sup_{\nu \in [N^{-1},1]} \big\| \big( \sum_{J \in \mathbb{I}_{\log(N)}} |P_J F|^2 (\tilde{w}^{r,\nu,J}_{\mu})^{\frac{1}{2}} \big)^{\frac{1}{2}} \big\|_{L^4_{t,x}}^4.
\end{equation}
Let $n = \log(N)$ for brevity.
We apply Minkowski's inequality to find
\begin{equation*}
\big\| \big( \sum_{J \in \mathbb{I}_{n}} |P_J F|^2 (\tilde{w}^{r,\nu,J}_{\mu})^{\frac{1}{2}} \big)^{\frac{1}{2}} \big\|_{L^4_{t,x}} \leq \big[ \sum_J \big( \int |P_J F|^4 \tilde{w}^{r,\nu,J}_{\mu} dx dt \big)^{\frac{2}{4}} \big]^{\frac{1}{2}}.
\end{equation*}
Since by Fourier support assumptions we have
\begin{equation*}
P_J \big( \sum_{k \in \Z/N} e^{i (kx + tk^2)} a_k \chi_{N^2}(x) \big) \chi_{N^2}(t) = \sum_{k \in 2J} a_k P_J \big( e^{i (kx + tk^2)} a_k \chi_{N^2}(x) \big) \chi_{N^2}(t)
\end{equation*}
and $\tilde{w}^{r,\nu,J}_{\mu}$ is still of product form, the estimate
\begin{equation}
\label{eq:SingleTermEstimate} 
\int |P_J F|^4 \tilde{w}^{r,\nu,J}_{\mu} dx dt \lesssim \sum_{k \in 2J} |a_k|^4 N^2 N^{2 \alpha'}
\end{equation}
is straightforward by plugging in the convex combinations given by $\tilde{w}^{r,\nu,J}_{\mu}$. Taking \eqref{eq:ExponentialSumScaling}, \eqref{eq:LocalSF}, and \eqref{eq:SingleTermEstimate} together we find
\begin{equation*}
\| e^{it \Delta} P_N f \|_{L^4_{t,x}(\mathcal{I} \times \T)} \lesssim \log(N)^c N^{-\frac{1}{4}+\frac{\alpha'}{2} - \frac{\alpha}{4}} \big( \sum_{k \in \Z} |a_k|^2 \big)^{\frac{1}{2}},
\end{equation*}
as claimed.
\end{proof}
The proof of Theorem \ref{thm:FractalL6Strichartz} follows along the same argument and is omitted for this reason. 

\bigskip

\emph{Statement on AI use.}
At a later stage of preparing the manuscript, for checking references, consistency, and for the present formulation of Lemma 2.3, ChatGPT-6.1 Sol was used. All remaining mathematical content is entirely human generated. The author takes full responsibility for the contents of this paper.

\section*{Acknowledgements}
Financial support by the Humboldt Foundation through a Feodor-Lynen return fellowship is gratefully acknowledged.

\end{document}